\documentclass[preprint,12pt]{elsarticle}

\usepackage{amssymb}
\usepackage{graphicx}
\usepackage{euscript,color}
\usepackage{amssymb,amsfonts,amsmath,amscd,dsfont,mathrsfs,amsthm}

\newtheorem{thm}{Theorem}
\newtheorem{cor}{Corollary} %[subsection]
\newtheorem{lem}{Lemma} %[subsection]
\newtheorem{dfn}{Definition}
\newtheorem{prop}{Proposition}

\def\x{{\mathbf x}}

 \def\E{\mathds E}
 \def\P{\mathds P}

\journal{Applied and Computational Harmonic Analysis}

\begin{document}

\begin{frontmatter}



\title{{\bf Suboptimality of Nonlocal Means} \\ {\bf for Images with Sharp Edges}}


\author{Arian Maleki\corref{cor1}\fnref{fn1}} \ead{arian.maleki@rice.edu} \ead[url]{http://www.ece.rice.edu/~mam15/}
\author{Manjari Narayan\fnref{fn2}} \ead{manjari@rice.edu} \ead[url]{http://dsp.rice.edu/dspmember/12}
\author{Richard G.\ Baraniuk\corref{cor2}\fnref{fn3}} \ead{richb@rice.edu} \ead[url]{http://web.ece.rice.edu/richb/}

\address{Dept.~of Computer and Electrical Engineering, Rice University, MS-380 \\ 6100 Main Street, Houston, TX 77005, USA \\ }

\cortext[cor1]{Corresponding author}
\cortext[cor2]{Principal corresponding author}

\fntext[fn1]{Phone: +1 713.348.3579; Fax: +1 713.348.5685}
\fntext[fn2]{Phone: +1 713.348.2371; Fax: +1 713.348.5685}
\fntext[fn3]{Phone: +1 713.348.5132; Fax: +1 713.348.5685}

\begin{abstract}
We conduct an asymptotic risk analysis of the nonlocal means image denoising algorithm for the Horizon class of images that are piecewise constant with a sharp edge discontinuity.  We prove that the mean square risk of an optimally tuned nonlocal means algorithm decays according to $n^{-1}\log^{1/2+\epsilon} n$, for an $n$-pixel image with $\epsilon>0$. This decay rate is an improvement over  some of the predecessors of this algorithm, including the linear convolution filter, median filter, and the SUSAN filter, each of which provides a rate of only $n^{-2/3}$. It is also within a logarithmic factor from optimally tuned wavelet thresholding. However, it is still substantially lower than the the optimal minimax rate of $n^{-4/3}$.
\end{abstract}


\end{frontmatter}

% \linenumbers

\section{Introduction}
% Removed
% to bayesian approaches \cite{Portilla:2003bh,Crouse:1998ve}

%Here is the organization of the paper. In this section we formally state the nonlocal means algorithm, the criteria we use for measuring the performance of the algorithm and the model we consider for edges. Finally, we briefly review the contributions of this paper and mention the related work. Section \ref{sec:theory} will be devoted to the proofs of our main theorems. In Section \ref{sec:extension} we extend our results to other edge models.   

\subsection{Image denoising}

The long history of image denoising is testimony to its central importance in image processing.  A wide range of algorithms have been developed, ranging from simple linear convolution and median filtering to  total variation denoising \cite{RuOsFa92} and sparsity exploiting algorithms such as wavelet shrinkage \cite{Donoho:1998p1561}. 
Due to the sensitivity of human visual system to edges, the ability to preserve sharp edges is an important criterion for noise removal algorithms. Therefore Korostelev and Tsybakov proposed a framework to characterize the performance of image denoisers on edges \cite{Tsybakov:1993uq}.  Based on this framework, we aim to characterize the performance of several denoising algorithms that represent the current state of the art image enhancement techniques.  In particular, we will focus on the popular and powerful nonlocal means (NLM) algorithm.   

\subsection{The minimax framework}\label{ssec:framework}

In this paper, we are interested in estimating a function $f:[0,1]^2 \rightarrow \mathds{R}$ from noisy pixel level observations.  Define ${\rm Pixel}(i,j) = [\frac{i}{n},\frac{i+1}{n})\times [\frac{j}{n},\frac{j+1}{n})$, and let $x_{i,j}= {\rm Ave} (f \ | \ {\rm Pixel}(i,j))$ be the pixel level averages of $f$. We observe the samples
\[
y_{i,j} = x_{i,j} +z_{i,j},
\]
where, $z_{i,j}$ is iid $N(0,\sigma^2)$.  The goal is to recover the original pixel values $x_{i,j}$ from the observations $y_{i,j}$, based on some information about the function $f$.  For a given function $f$ and an estimator $\hat{f}$ we define the {\em risk function} as
\begin{equation}\label{eq:riskdef}
R_n(f, \hat{f}) = \E \left(\frac{1}{n^2} \sum_i \sum_j (x_{i,j}-\hat{f}_{i,j})^2\right).
\end{equation}
The risk can also be written as 
\begin{equation}\label{eq:biasvar}
R_n(f, \hat{f}) =  \left(\frac{1}{n^2} \sum_i \sum_j (x_{i,j}-\E \hat{f}_{i,j})^2\right) + \E \left(\frac{1}{n^2} \sum_i \sum_j (\hat{f}_{i,j}-\E \hat{f}_{i,j})^2\right),
\end{equation}
where the first and second terms correspond to the {\em bias} and {\em variance} of the estimator $\hat{f}$, respectively.

Let $f$ belong to a class of functions $\mathcal{F}$, e.g., a class of edge-like images that represent edges with different shapes and orientations. The risk defined in \eqref{eq:riskdef} depends on the specific choice of $f$. We define the risk of an estimator $\hat{f}$ on the class $\mathcal{F}$ as the risk of the worst-case signal, i.e.,
\[
 R_n({\mathcal{F}},\hat{f})= \sup_{f \in \mathcal{F}} R_n(f, \hat{f}).
\]
The minimax risk over functions in $\mathcal{F}$ is then defined as the risk of the best possible estimator, i.e., 
\[
R_n^{*}(F) = \inf_{\hat{f}} \sup_{f  \in F} R_n(f, \hat{f}).
\] 
The minimax risk is a {\em lower bound} for the performance of all measurable estimators for signals in $\mathcal{F}$. 
 
In this paper we are interested in the asymptotic setting where the number of pixels $n \rightarrow \infty$.  For all of the estimators we consider, $R_n(\mathcal{F}, \hat{f}) \rightarrow 0$ as $n \rightarrow \infty$. Therefore, we consider the {\em decay rate} of the risk as the performance measure.  We will derive the minimax risk for several popular image denoising techniques below.

We will use the following asymptotic notation in this paper.

\begin{dfn}
$f(n) = O(g(n))$ as $n \rightarrow \infty$, if and only if there exist $n_0$ and $c$ such that for any $n> n_0$, $|f(n)| \leq c |g(n)|$. Likewise, $f(n) = \Omega(g(n))$ as $n \rightarrow \infty$, if and only if there exist $n_0$ and $c$ such that for any $n> n_0$, $|f(n)| \geq c |g(n)|$. Finally, $f(n) = \Theta(g(n))$, if $f(n) = O(g(n))$ and $f(n)= \Omega(g(n))$. We may interchangeably use $f(n) \asymp g(n)$ for $f(n) = \Theta(g(n))$.
\end{dfn}

\begin{dfn}
$f(n) = o(g(n))$ if and only if $\lim_{n \rightarrow \infty} \frac{f(n)}{g(n)} = 0$.
\end{dfn}

\subsection{Horizon edge model}\label{ssec:edgemodel}

Several different image edge models have been developed in the image processing and denoising literature. Here we will use the {\em Horizon model} that contains piecewise constant images with edges that are smooth in the direction of the edge contour but discontinuous in the direction orthogonal to the edge contour \cite{Tsybakov:1993uq,Donoho:1999p1950}. Specifically, let ${\sl H\ddot{o}lder}^{\alpha}(C)$ be the class of H${\rm \ddot{o}}$lder smooth functions on $\mathds{R}$, defined as follows: $h \in H\ddot{o}lder^{\alpha}(C)$ if and only if
%For $0 < \alpha <1$, $f \in Holder^{\alpha}(C)$ if and only if
%\[
%|h(x)-h(y)| \leq C|x-y|^{\alpha}. 
% \]
%For $\alpha>1$, the function will be $k= \lfloor \alpha \rfloor$ times differentiable and,
\[
|h^{(k)}(t_1)-h^{(k)}(t'_1)| \leq C|t_1-t'_1|^{\alpha-k},
\]
where $k = \lfloor \alpha \rfloor $.  Given a one-dimensional smooth edge contour function $h$, we define the image $f_h: [0,1]^2 \rightarrow \mathds{R}$ as $f_h(t_1,t_2) = \textbf{1}_{\{t_2 < h(t_1)\}}$, where $\textbf{1}_{\{ \cdot\}}$ is the indicator function. Based on this construction, we define the Horizon class of functions as
\begin{align}
H^{\alpha} (C) = \{f_h(t_1,t_2) : h \in  H\ddot{o}lder^{\alpha}(C) \cap H\ddot{o}lder^{1}(1)  \},
\end{align}
where $\alpha$ is the smoothness of the edge contour. Figure \ref{fig:horizon.pdf} plots a representative function from this class.

\begin{figure}
\begin{center}
\includegraphics[width=8cm]{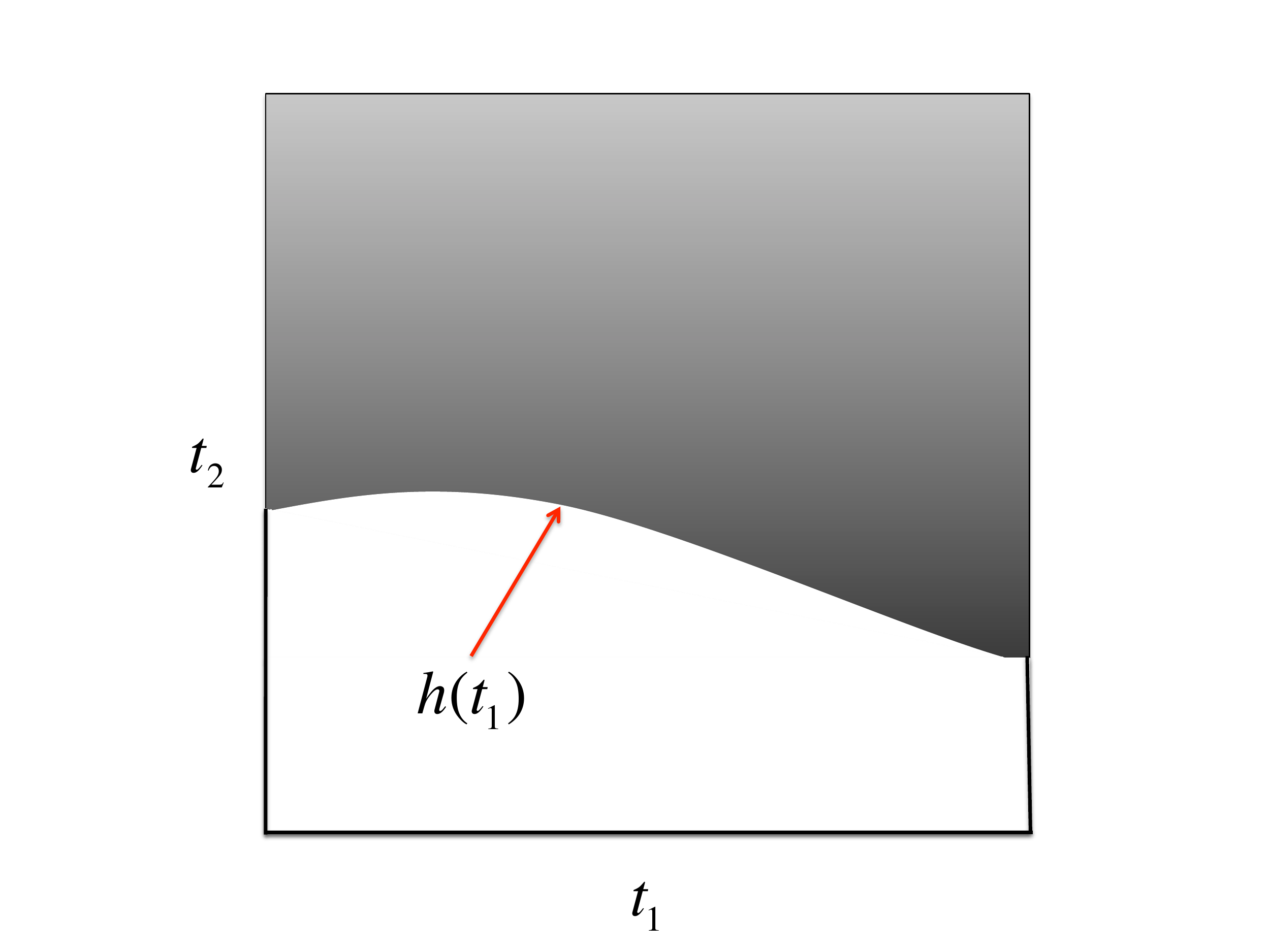}
\caption{An example of a Horizon function, a piecewise constant image containing an edge that is $H\ddot{o}lder^{\alpha}$ smooth in the direction of the edge contour but discontinuous in the direction orthogonal to the edge contour.} 
  \label{fig:horizon.pdf}
  \end{center}
\end{figure}

The following theorem, proved in \cite{Tsybakov:1993uq}, specifies the minimax risk of the class of all measurable estimators on $H^{\alpha} (C)$.

\begin{thm}\label{thm:minimax}
{\rm \cite{Tsybakov:1993uq}} For $\alpha \geq1$, the minimax risk of the class $H^{\alpha} (C)$ is
\begin{equation}\label{eq:minimax}
R_n^{*}(H^{\alpha} (C)) \asymp n^{\frac{-2\alpha}{\alpha+1}}. 
\end{equation}
\end{thm}

We are particularly interested in the case $\alpha =2$ edges, for which the optimal rate is $n^{-4/3}$.
The rate provided in the above theorem is the Holy Grail of image denoising algorithms. 

\subsection{A menagerie of denoising algorithms}\label{ssec:algorithm}

We will perform a minimax risk analysis of not just nonlocal means but a number of other popular image denoising algorithms.  

\subsubsection{Linear filtering} \label{sssec:lineardef}

The classical denoising method is the \textit{linear convolution filter}, which estimates the image via
\begin{eqnarray}\label{eqn:linearfiltereq1}
\hat{f}^{LF}_g(i,j) =\sum_{m} \sum_{\ell} g_{m,\ell} y_{i-m,j-\ell},
\end{eqnarray}
where $g$ is a two dimensional filter impulse response that satisfies $\sum_i \sum_j g_{i,j}=1$.\footnote{For simplicity of analysis, we use a periodic extension of $y$ at the image boundaries.}
When all the weights $g_{ij}$ are equal, the algorithm is called the \textit{running average} or the \textit{box filter}. Most of the linear filters used in practice are symmetrical and approximately isotropic.

 \begin{dfn}
Let $g$ be a real and symmetric filter response, i.e., $g_{i,j} = g_{-i,j}= g_{i,-j}$, and let $\hat{G}(\omega_1, \omega_2)$ represent its two-dimensional Fourier transform. The filter is {\em isotropic} if and only if there exists a function $F:\mathds{R} \rightarrow \mathds{C}$, such that
\[
\hat{G}({\omega_1,\omega_2}) = F\left(\sqrt{\omega_1^2+\omega_2^2}\right)  \ \ \ \forall -\pi < \omega_1, \omega_2 \leq \pi.
\]
\end{dfn}  

Isotropic filters are popular, because they treat image features similarly regardless of their directions. Let ${\rm grad}(\cdot)$ be the gradient operator. The following theorem, proved in Section \ref{ssec:linearproof}, provides the decay rate of the risk of linear convolution. 

\begin{thm}\label{thm:linear}\sloppy
Consider the linear convolution filter \eqref{eqn:linearfiltereq1} and suppose that $g$ is real, symmetric, and isotropic. Furthermore, assume that $\| {\rm grad}(\hat{G}(w_1,w_2))\|_2 \leq C$ for a fixed constant $C$. Then,
\[
\inf_{g} \sup_{f \in {H}^{\alpha}(C)} R_n(f, \hat{f}^{LF}_g) \asymp n^{-\frac{2}{3}}.
\]
\end{thm}

Castro and Donoho \cite{CaDo09} have proved a similar result for the special case of the box filter.  While the Horizon model used in \cite{CaDo09} is slightly different from our model, their proof works in our setting as well.

\subsubsection{Yaroslavsky / SUSAN filter}\label{sssec:susandef}

While linear filters are popular in image processing due to their simplicity, they unfortunately blur images with sharp edges. One popular alternative is to adapt the weight of each pixel in the average (\ref{eqn:linearfiltereq1}) according to the distance between its noisy value and the value of the pixel we aim to estimate. Let $\mathcal{C}_{i,j}^{\Delta_n} \triangleq \{(m,\ell) \ | \     i- \Delta_n \leq m \leq i+ \Delta_n,   j- \Delta_n \leq \ell \leq j+ \Delta_n\}$ denote the $\Delta_n$-neighborhood of the pixel $(i,j)$. One popular approach for setting the weights is
 \[
 w^Y_{i,j}(m,\ell) = {\rm e}^{-\frac{(y_{m,\ell} - y_{i,j})^2}{2\tau^2}},
 \]
 from which we calculate the estimate 
\begin{equation}\label{eq:susan1}
 \hat{f}^Y_{\Delta_n, \tau}(i,j) = \frac{\sum_{m =i-\Delta_n}^{i+\Delta_n} \sum_{\ell = j-\Delta_n}^{j+\Delta_n} w^Y_{i,j}(m,\ell) \, y_{m,\ell}}{\sum_{m =i-\Delta_n}^{i+\Delta_n} \sum_{\ell = j-\Delta_n}^{j+\Delta_n} w^Y_{i,j}(m,\ell)}.
 \end{equation}
%where $\Delta_n$ is the size of the neighborhood of pixel $(i,j)$. 
Only the pixels in the $\Delta_n$-neighborhood of $(i,j)$ contribute to the estimate of that pixel.  This algorithm is called the {\em Yaroslavsky Filter}  (YF) or {\em SUSAN filter} \cite{Yaroslavsky85,SmBr97}; slight modifications are known as the {\em bilateral filter} \cite{ToMa98} and {\em $\sigma$-filter} \cite{Lee83}.
 
%In this paper we analyze this algorithm as well. First,  since the linear filters are a special form of YF, based on the results of the last section we can conclude conclude the following theorem. 
%\begin{thm}\label{thm:uppersf}
%The risk of YF filter satisfies,
%\[
%  \inf_{\tau, \Delta_n}  \sup_{f \in \mathcal{H}^{\alpha}(C_{\alpha})}R(f, \hat{f}^Y) = O(n^{-2/3}).
%\]
%\end{thm}
%But again it is natural to ask if the YF can outperform the LF. To prove a lower bound we introduce the following oracle form
%of the YF algorithm. 

To calculate the risk of the YF, we consider a slightly different, oracle-based algorithm. Suppose that in setting the weights $w_{i,j}^Y(m,\ell)$ of the YF we have access to the actual (and not the noisy) value of the pixel $(i,j)$. Using this oracle information we can set the weights according to   
\[
 w^{SY}_{i,j}(m,\ell) = {\rm e}^{-\frac{(y_{m,\ell} - x_{i,j})^2}{2\tau^2}}.
\]
Intuitively, the oracle weights are ``less noisy'' than the actual filter weights.  Plugging these weights into \eqref{eq:susan1} we obtain what we call the {\em semi-oracle Yaroslavsky filter} (SYF). The following theorem, proved in Section \ref{ssec:susanproof}, shows that, as far as the decay rate is concerned, the SYF's performance is the same as the linear filter and box filter. 
 
\begin{thm}\label{thm:lowersusan}
The risk of SYF algorithm satisfies
\[
\inf_{\tau, \Delta_n} \sup_{f \in {H}^{\alpha}( C)} R_n(f, \hat{f}^{SY}) = \Omega(n^{-2/3}).
\]
\end{thm}
 
\subsubsection{Sparsity based denoising}\label{sssec:waveletdef}
 
Another popular class of image denoising methods exploit sparsity in some transform domain via thresholding.  Wavelets are often used as the sparsity domain for natural images. Let $\mathcal{W}(y)$ represent the separable two-dimensional wavelet transform of the image, let $\mathcal{IW}$ represent the inverse wavelet transform, and let $\mathcal{T}$ be the hard thresholding function, i.e.,  $\mathcal{T}_{\theta}(x)= x \textbf{1}_{\{|x|> \theta\}}$. Then wavelet thresholding denoising corresponds to
\[
\hat{f}^{W}_{\theta} = \mathcal{IW}(\mathcal{T}_{\theta}(\mathcal{W}(y))). 
\]
Donoho and Johnstone have proven that $\sup_{f \in {H}^{\alpha} ( C)} R_n(f,\hat{f}^W)= \Omega(n^{-1})$ \cite{Donoho:1999p1950}, \cite{Mal97}.  Even though this rate is an improvement over the above algorithms, is still far from the optimal achievable rate of $n^{-\frac{4}{3}}$ for $\alpha= 2$.

This suboptimality spurred  the development of other sparsity-inducing transformations, including curvelets \cite{Donoho:2000p1676}, wedgelets \cite{Donoho:1999p1950}, shearlets \cite{KuLa07}, and contourlets \cite{DoVe05}. Among these transforms, wedgelet denoising provably achieves the optimal rate of $n^{-\frac{4}{3}}$ for $\alpha= 2$ \cite{Donoho:1999p1950}. However, wedgelet denoising performs poorly on textures, which has limited its application in practice to date.

\section{Nonlocal means denoising} \label{sssec:nlmdef}

%The success of the YF in practice motivated the researchers to exploit the structures present in the noisy observation more. 

The YF estimator sets its weights according the noisy pixel values and their spatial vicinity; however neither of these two features are reliable for noisy, edgy images.  In contrast, the {\em nonlocal means} (NLM) algorithm sets its weights according to the proximity of the image patch surrounding each noisy pixel with other patches in the image \cite{Buades:2005p27}. Define the $\delta_n$-neighborhood distance $d_{\delta_n}(y_{i,j},y_{m,\ell})$ between two  observations as
\begin{align}
d^2_{\delta_n} (y_{i,j}, y_{n,p}) =& \frac{1}{\rho_n^2}\sum_{m = -\delta_n}^{\delta_n} \sum_{\ell = -\delta_n}^{\delta_n} |y_{i+\ell, j+ m} - y_{n+\ell,p+m}|^2
-|y_{i,j}-y_{n,p}|^2, \nonumber
 \end{align}
where $\rho_n^2 = (2\delta_n+1)^2-1$.  Note that, in contrast to the definition in \cite{Buades:2005p27}, we have removed the center element $|y_{i,j}-y_{n,p}|^2$ from the summation. Since we assume that $\delta_n \rightarrow \infty$  as $n \rightarrow \infty$, the effect is negligible on the asymptotic performance.  But, as we will see in Section \ref{sec:theory}, removing the center element simplifies the calculations considerably. NLM uses the neighborhood distances to estimate
\begin{align} \label{eq:nlm1}
\hat{f}^{N}_{i,j} = \frac{\sum_{(m,\ell) \in S} w^N_{i,j}(m,\ell) y_{m,\ell} } {\sum_{(m,\ell) \in S} w^N_{i,j}(m,\ell)},
\end{align}
where $S =\{1,2,\ldots,n\} \times \{1,2,\ldots,n\}$ and $w_{i,j}(m,\ell)$ is set according to the $\delta_n$-neighborhood distance between $y_{i,j}$ and $y_{m,\ell}$. For the simplicity of notation, in cases where both the reference pixel $(i,j)$ and the algorithm are obvious from the context, we will omit the superscript and subscript of the weight and use the simplified notation $w_{m,\ell}$ instead of $w^N_{i,j}(m,\ell)$.  
It is straightforward to verify that $\E(d^2_{\delta_n} (y_{i,j},y_{m,\ell})) = d^2_{\delta_n} (x_{i,j},x_{m,\ell}) + 2\sigma^2$, which suggests the following strategy for setting the weights:
\begin{eqnarray}
w^N_{i,j}(m,\ell) =\! \! \left\{\begin{array}{rl}
 1 &  \mbox{ if $d^2_{\delta_n} (y_{i,j},y_{m,\ell}) \leq 2\sigma^2+ t_{n} $,}\\
 0 &   \mbox{ otherwise,}
\end{array}\right.
\end{eqnarray}
where $t_n$ is the {\em threshold parameter}. Soft/tapered versions of setting the weights have been explored and are often used in practice \cite{Buades:2005p27}. However, the above untapered weights capture the essesnse of the algorithm while simplifying the analysis.  We postpone the discussion of tapered weights until Section \ref{app:taperedweights}. 
  
There are two main differences between the NLM and YF algorithms.  First, the pixels that contribute in the NLM averaging are not necessarily in the local neighborhood of the reference pixel (hence the monicker ``nonlocal''). Second, the NLM weights depend not on the difference between the pixel values but on distance between the pixel neighborhoods. In other words the pixel neighborhood is even more important than the pixel value. 

To derive a lower bound for the risk of NLM, we will analyze two algorithms that set the weights using some degree of oracle information regarding the true value of the signal. 
The {\em full oracle} NLM (FNLM) has access to $\E(d^2_{\delta_n} (y_{i,j},y_{m,\ell}))$ in setting the weights $w_{m,l}$ in (\ref{eq:nlm1}) and thus sets them using the noise-free values of the pixels
\begin{eqnarray}
w^F_{i,j}(m,\ell) =\! \! \left\{\begin{array}{rl}
1 &  \mbox{ if $d^2_{\delta_n} (x_{i,j},x_{m,\ell}) \leq t_{n,} $,}\\
0 &  \mbox{ otherwise.}
\end{array}\right.
\end{eqnarray}
The {\em semi-oracle} NLM (SNLM) differs only slightly from the standard NLM in that it uses the semi-oracle neighborhood distance  
\begin{align} \label{eq:$NL_1$}
\bar{d}^2_{\delta_n} (y_{i,j}, y_{n,p}) \triangleq & \frac{1}{\rho_n^2} \left( \sum_{m = -\delta_n}^{\delta_n} \sum_{\ell = -\delta_n}^{\delta_n} |x_{i+\ell, j+ m} - y_{n+\ell,p+m}|^2 -(x_{i,j} - y_{n,p})^2\right),
\end{align}
and then sets the weights in (\ref{eq:nlm1}) according to
 \begin{eqnarray} \label{eq:$NL_2$}
w^S_{i,j}(m,\ell) =\! \! \left\{\begin{array}{ll}
1 &  \mbox{ if $\bar{d}^2_{\delta_n} (y_{i,j},y_{m,\ell}) \leq \sigma^2+ t_{n,} $,}\\
0 &  \mbox{ otherwise.}
\end{array}\right.
\end{eqnarray}
Unlike FNLM, SNLM assumes that just one-half of the noise is removed from the distance estimates. Therefore, the distances calculated in the SNLM are more accurate than the standard NLM but less accurate than in the FNLM. 
In the rest of the paper, we will use $\hat{f}^{N}$, $\hat{f}^{S}$, and $\hat{f}^{F}$ to denote the NLM, SNLM, and FNLM estimators, respectively.  

% Section Criteria

\section{Main Results} \label{sec:mainresults}

Our first result, proved in Section \ref{ssec:upperbound}, establishes an {\em upper bound} on the risk of NLM. 

\begin{thm} \label{thm:upperbound} Fix $\epsilon >0$ and consider NLM denoising with $\delta_n = 2\log^{\frac{1}{2}+ \epsilon} n$ and $t_n = \frac{2\sigma^2}{{\log^{\frac{\epsilon}{2}} n}}$. The risk of this algorithm over the class $H^{\alpha} (C)$ is
\begin{equation}\label{eq:upperboundmain}
\sup_{f \in H^{\alpha} (C)}R(f, \hat{f}^{N}) = O\left(\frac{\log ^{\frac{1}{2}+\epsilon} n}{n}\right).
\end{equation}
\end{thm}

Before we discuss the implications of this theorem, it is important to note that, while we can improve the decay rate as close as we desire to $O(n^{-1}{\log ^{\frac{1}{2}}n })$, the constants that are involved in the big-$O$ notation grow as $\epsilon$ decreases. Therefore, in practice very small values of $\epsilon$ are not desirable. 

Comparing the upper bound \eqref{eq:upperboundmain} with the optimal minimax risk \eqref{eq:minimax} indicates that NLM is {\em suboptimal} for $\alpha>1$.  In other words, NLM cannnot exploit the smoothness of edge contours in images. 
 
The bound in Theorem \ref{thm:upperbound} is for a specific choice of parameters, and it is natural to ask whether NLM can achieve the optimal rate with some other choice of parameters.  To answer this question, we consider SNLM, which outperforms standard NLM in general. We make the following mild assumptions:  
\begin{itemize}

 \item[A1:] The window size $\delta_n \rightarrow \infty$ as $n \rightarrow \infty$. This assumption is critical to ensuring good performance of any NLM estimator. 

\item[A2:] The threshold is set to $\sigma^2+ t_n$ as explained in (\ref{eq:$NL_2$}) with $t_n>0$. This ensures that if the neighborhood of pixel $(m,\ell)$ is exactly the same as the neighborhood of pixel $(i,j)$, then $w_{m,\ell}=1$ with high probability.  

\item[A3:] The threshold $t_n$ is set such that, if the noise-free neighborhoods are different in more than half of their pixels, i.e., if $d^2(x_{i,j},x_{m,\ell}) \geq \frac{1}{2}$, then $\P(w^F_{i,j}(m,\ell)=1) = o\left(n^{-1} \right)$.  

\item[A4:] $\delta_n = O(n^{\beta})$, for some $\beta\leq 0.3$.

\end{itemize}

The following theorem provides a {\em lower bound} on the performance of SNLM.

\begin{thm} \label{thm:lowerbound}
Suppose that $\delta_n $ and $t_n$ satisfy A1--A4. The risk of the SNLM over the class $ {H}^{\alpha}( C)$ is
\[
\inf_{\delta_n, t_n} \sup_{f \in {H}^{\alpha}( C)}  R(f,\hat{f}^{S}) = \Omega (n^{-1}).
\]
\end{thm}

This bound is still suboptimal compared to the $n^{-4/3}$ minimax rate for $\alpha =2$.
In the words of John Cornyn III, the junior United States Senator for Texas, 
``The problem with a mini-deal is we have a maxi-problem'' \cite{Broder11}.

Remarkably, this lower bound is achieved on a very simple image on which NLM would be assumed to work very well: $\textbf{1}_{\{t_2<0.5\}}$ (see Figure \ref{fig:horizedge}).  Here is what goes wrong.  Consider the estimation of an ``edge"  pixel $(i,j)$ that satisfies $j = \lceil nh(\frac{i}{n}) \rceil$.  Define the set $J = \{(m,\ell) \ | \ \ell = \lfloor n h(\frac{m}{n})\rfloor \}$ as the set of pixels just below the edge. We will prove the probability that a pixel in $J$ contributes to the NLM estimate ($w_{i,j}(m,\ell) =1$) is larger than $p_0$, where $p_0$ does not depend on $n$. This happens due to the low ``signal to noise ratio'' in the distance estimates. 
Hence $\Theta(n)$ pixels of $J$ will contribute to the NLM estimate. Since these pixels have $x_{m,\ell} =1$, they introduce a large bias in the estimate.  In fact, we show below that the bias, as defined in \eqref{eq:biasvar}, will be larger than $\frac{np_0}{n+np_0+np_0}$.  Here $np_0$ corresponds to the pixels below the edge that pass the threshold. This shows that the bias is clearly $\Theta(1)$. Since there are $n$ edge pixels, the risk of the estimator over the entire image is $\Omega(n^{-1})$.
%To understand the above theorem consider a very simple image model shown in Figure \ref{fig:horizontalline}. This figure is a representation of an image $f_h(t_1,t_2) = 1_{\{t_2<h(t_1)\} }$ for $h(t_1) = \frac{1}{2}$.We expect the edge points to be very problematic for the nonlocal means algorithms. Let $n$ be an even number so that there are $n$ pixels exactly on the edge. If we consider the pixels just below the edge it is clear that the difference $\hat{d}(x_{i,n/2}, x_{m, n/2-1}) = \Theta(\frac{1}{2\delta_n+1})$. This is very small, compared to the variance of the noise which is present in the neighborhood distance estimation. Therefore we will prove that the probability that these pixels contribute in the average of the nonlocal means is larger than some $p_o$ that does not depend on the $n$. The formal statement of the above argument can be found in Theorem \ref{thm:belowedge}. This phenomena leads to Theorem \ref{thm:lowerbound}. There are many missing steps in the above intuitive argument that will be explained in Section \ref{ssec:lowerbound}.

\begin{figure} 
\begin{center}
  \includegraphics[width=6cm]{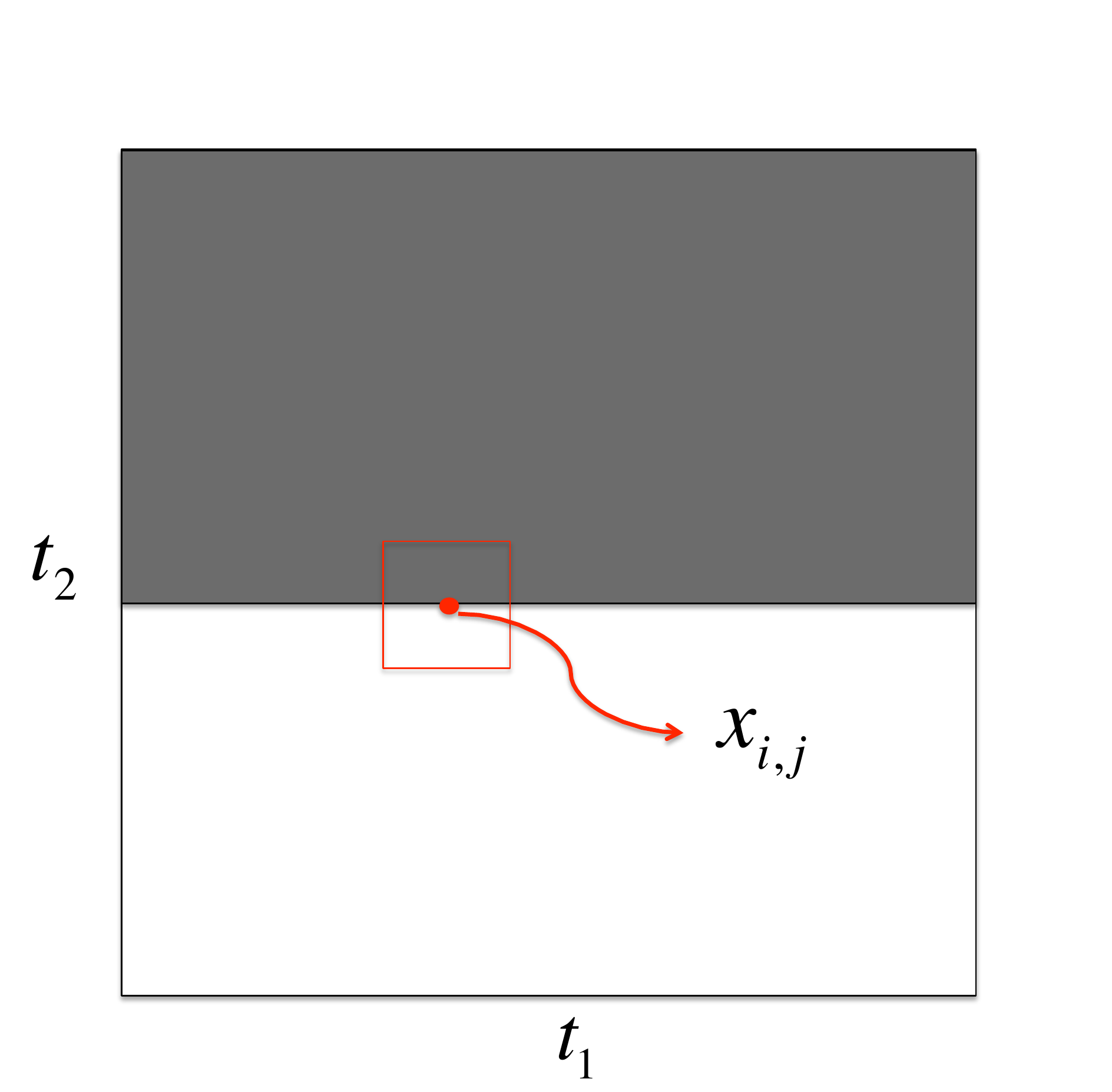}\\ 
  \caption{The simple image Horizon $\mathbf{1}_{\{ t_2 < 0.5\}}$ used for proving the various lower bounds.}
  \label{fig:horizedge}
  \end{center}
\end{figure}

\section{Proofs of the Main Theorems}\label{sec:theory}
%In this section we prove our main results.  Section \ref{ssec:linearproof} will be devoted to the proof of Theorem \ref{thm:linear}.  We prove Theorem \ref{thm:lowersusan} in Section \ref{ssec:susanproof}. Section \ref{ssec:upperbound} establishes the upper bound we mentioned in Theorem  \ref{thm:upperbound} for the
%NLM and finally, Section \ref{ssec:lowerbound} summarizes the proof of Theorem \ref{thm:lowerbound}.

\subsection{Proof of Theorem \ref{thm:linear}} \label{ssec:linearproof}

The proof has two main steps. The first step is to prove that there exists a linear filter for which the supremum risk is upper bounded by $O(n^{-2/3})$. For this step we use
Theorem 3.1 and 3.2 from \cite{CaDo09}, which establish the same upper bound for the box filter. The second and more challenging step is to prove that no other linear filter
can improve on this decay rate. The rest of this section is dedicated to the proof of this fact. 

Consider the function $f_h(t_1,t_2)$ for $h(t)=\frac{1}{2}$ and suppose that $n$ is even. This function is displayed in Figure \ref{fig:horizedge}. Let
\[
{X}(k_1, k_2) = \frac{1}{n} \sum_{\ell_1} \sum_{\ell_2} x_{\ell_1, \ell_2} {\rm e}^{-j \frac{2\pi k_1 \ell_1}{n}} {\rm e}^{-j \frac{2\pi k_2 \ell_2}{n}}
\]
represent the Discrete Fourier Transform (DFT) of a discrete two-dimensional discrete signal $x$. Since $y = x + z$, the DFT of $\hat{f}^{LF}_g$ equals
\[
{\hat{F}}^{LF}_{g}(k_1,k_2) = Y(k_1,k_2) \cdot G(k_1,k_2) = X(k_1,k_2) G(k_1,k_2) + Z(k_1,k_2)G(k_1,k_2),
\]
where $Z$ is again iid $N(0, \sigma^2)$. For $f_h(t_1,t_2)$ with $h(t_1)= \frac{1}{2}$, $X(k_1,k_2)$ satisfies
\begin{eqnarray}
X(k_1,k_2) =\! \! \left\{\begin{array}{ll}
 \!  0 & \mbox{ if $k_1 \neq 0$, }\\
\!  \frac{1- {\rm e}^{-j \pi k_2}}{1- {\rm e}^{-j \frac{2\pi k_2}{n}}}  &   \mbox{ if $k_1 = 0$.}
\end{array}\right.
\end{eqnarray}
It is easy to see that $R_n(f, \hat{f}^{LF}_g) = \frac{1}{n^2}\E(\|X- \hat{F}^{LF}_g \|_F^2)$, where $\| Y \|_F^2 \triangleq \sum_{k_1,k_2} |Y(k_1,k_2)|^2$.
If we define $B(\hat{f})$ as the bias of the estimator $\hat{f}$, then we have
\begin{eqnarray*}
 {\rm B}^2(\hat{f}^{LF}_g) &=& \frac{1}{n^2}\sum_{1 \leq k_2 \leq n, \ {\rm odd}} |1- G(0,k_2)|^2 \frac{1}{ \sin^2 \frac{\pi k_2}{n}}  \\
&{\geq}& \frac{1}{n^2}\sum_{1 \leq k_2 \leq n^{2/3}, \ {\rm odd}} |1- G(0,k_2)|^2 \frac{1}{ \sin^2 \frac{\pi k_2}{n}}.
\end{eqnarray*}
The variance of the estimator is
\[
{\rm Var}(\hat{f}^{LF}_g) = \frac{1}{n^2}\sum_{k_1,k_2}  |G(k_1,k_2)| ^2 \sigma^2.
\]
We know that
\begin{eqnarray}\label{eq:dis_cont}
\frac{1}{n^2} \sum_{k_1,k_2}  |G(k_1,k_2)| ^2 &=& \int \int |\hat{G}(\omega_1,\omega_2)|^2+ O(n^{-1}),
\end{eqnarray}
where $\hat{G}$ is the continuous Fourier transform of $g$ and satisfies $\|{\rm grad}(\hat{G})\|_2 \leq C$.
Since $g$ is isotropic, there exists $F : \mathds{R^2} \rightarrow C$ such that
\[
\hat{G}(\omega_1, \omega_2) = F\left(\sqrt{\omega_1^2+\omega_2^2}\right).
\]
Changing the variables of integration in \eqref{eq:dis_cont} to polar coordinate radius $\omega_r = \sqrt{\omega_1^2+\omega_2^2}$ angle $\theta$, we have
\begin{eqnarray}\label{eq:int:cont_polar}
\int \int |\hat{G}(\omega_1,\omega_2)|^2 \geq 2 \pi \int_{r =0}^{2 \pi} r|F(r)|^2dr = 2 \pi \int_{\omega_2 =0}^{2 \pi} \omega_2 |\hat{G}(0, \omega_2)|^2d\omega_2.
\end{eqnarray}
Combining \eqref{eq:dis_cont} and \eqref{eq:int:cont_polar} we have
\[
{\rm Var}(\hat{f}^{LF}_h)=\frac{1}{n^2} \sum_{k_1,k_2}  |G(k_1,k_2)| ^2 \sigma^2 = \frac{4 \pi^2}{n^2}\sum_{k_2} k_2 |G(0,k_2)| ^2 \sigma^2- O(n^{-1}).
\]

Summing the lower bounds for the bias and variance of this estimator, we obtain the following lower bound for the risk of linear filtering:
\begin{eqnarray*}
\lefteqn{ R_n(f, \hat{f}^{LF})=  {\rm B}^2(\hat{f}^{LF}_g) + {\rm Var}(\hat{f}^{LF}_g)}  \\
&\geq & \frac{1}{n^2} \sum_{1 \leq k_2 \leq n^{2/3}, \ {\rm odd}} \! \! \! \! \! |1- G(0,k_2)|^2 \frac{1}{ \sin^2 \frac{\pi k_2}{n}} +  \frac{4 \pi^2}{n^2}\sum_{k_2} k_2 |G(0,k_2)|^2 \sigma^2- O(n^{-1})  \\
& = & \frac{1}{n^2}  \sum_{1 \leq k_2 \leq n^{2/3}, \ {\rm odd}} \! \! \! \! \! |1- G(0,k_2)|^2 \frac{n^2}{ \pi^2 k_2^2} +  \frac{4 \pi^2}{n^2} \sum_{k_2} k_2  | G(0,k_2)|^2 \sigma^2- O(n^{-1}).
\end{eqnarray*}
Minimizing the dominant term of the lower bound over the filter weights provides ${G}^*(0,k_2) = \frac{1}{1+ \frac{4\pi^4 \sigma^2 k_2^3}{n^2}}$ for odd values of $k_2$ and zero for even values of $k_2$. To find a lower bound we calculate the bias term with these optimal weights:
\begin{eqnarray*}
{\rm B}^2(\hat{f}^{LF}_{g^*}) &=&  \frac{1}{n^2} \sum_{1 \leq k_2 \leq n^{2/3}, \  {\rm odd}}  |1- G(0,k_2)|^2 \frac{n^2}{ \pi^2 k_2^2} \\
&=&  \frac{1}{n^2}  \sum_{1 \leq k_2 \leq n^{2/3}, \ {\rm odd}} \left(\frac{4\pi^4 \sigma^2 k_2^3/n^2}{1+4\pi^4 \sigma^2 k_2^3/n^2} \right)^2 \frac{n^2}{ \pi^2 k_2^2} \\
&\geq&  \frac{1}{n^2}  \sum_{1 \leq k_2 \leq n^{2/3}, \ {\rm odd}} \left(\frac{4\pi^4 \sigma^2 k_2^3/n^2}{1+4\pi^4 \sigma^2}\right)^2  \frac{n^2}{ \pi^2 k_2^2} \\
&=& \frac{1}{n^4} \left(\frac{4\pi^4 \sigma^2}{1+4\pi^4 \sigma^2} \right)^2  \sum_{1 \leq k_2 \leq n^{2/3}, \ {\rm odd}} k_2^4 \\
&=& \left(\frac{4\pi^4 \sigma^2}{1+4\pi^4 \sigma^2} \right)^2 \left( \frac{n^{-2/3}}{40} + o(n^{-2/3})\right).
\end{eqnarray*}
This completes the proof.

%\begin{thm}
%\cite{CaDo09} Consider the BF algorithm with window size $n^{4/3}$. The risk of this estimator satisfies
%\[
%\sup_{f \in \mathcal{H}^{\alpha}( C_{\alpha})} R_n(f, \hat{f}^{BF}_{\Delta_n})  \asymp n^{\frac{-2}{3}}.
%\]
%\end{thm}
%We finish this section with the example of ideal filters.  
%\begin{prop}
%Consider the LF algorithm and suppose that $h$ is real, symmetric, and isotropic. Furthermore assume that there exist $\omega_c$ such that
% \begin{eqnarray*}
%%
%\hat{H}(k_1,k_2) =\! \! \left\{\begin{array}{ll}
% \!  0 & \mbox{ if $k_1^2+k_2^2 \geq w_c^2$,}\\
%\! 1  &   \mbox{{\rm otherwise.}} 
%\end{array}\right. 
%%
%\end{eqnarray*}
%Then the risk of the estimator under this constraint satisfies,
%\[
%\inf_{\omega_c}\sup_{f \in \mathcal{H}^{\alpha}( C_{\alpha})} R_n(f, \hat{f}^{LF})  \geq \Omega \left(n^{\frac{-1}{2}}\right).
%\]
%\end{prop}
%The proof of this proposition is very similar to the proof of Theorem \ref{thm:linear} and therefore is skipped here. This Proposition shows that ideal low-pass filters 
%perform worse than the box filters on denoising sharp edges. This is intuitive, since the edges have important high frequency components. 

\subsection{Proof of Theorem \ref{thm:lowersusan}} \label{ssec:susanproof}

In this section, denote the pixel to be estimated as $x_{i,j}$. For clarity we use the notation $w_{m,\ell}$ instead of $w^{SY}_{i,j}(m,\ell)$. We first characterize some of the properties of the SYF weights. 

\begin{lem}\label{lem:lowerbd_expec}
Suppose that $x_{i,j}=0$. If $x_{m,\ell} = x_{i,j}$, then $\E(w_{m,\ell} y_{m,\ell}) =0$. Furthermore, if $|x_{i,j} -x_{m,\ell}|=1$, then $\E(w_{m,\ell} y_{m,\ell}) > \frac{\tau}{\sqrt{\sigma^2 + \tau^2}} {\rm e}^{\frac{-1}{2(\sigma^2+\tau^2)}}$.
\end{lem}

\begin{proof}
The first claim is clear from symmetry. To prove the second claim, we observe that $\E(w_{m,\ell} y_{m,\ell}) = \E(w_{m,\ell} x_{m,\ell}) + \E(w_{m,\ell}z_{m,\ell})$. Since $x_{m,\ell} = 1$, we calculate $\E(w_{m,\ell} )$ and $\E(w_{m,\ell}z_{m,\ell})$. It is clear that $\E(w_{m,\ell}z_{m,\ell}) =  \frac{1}{\sigma \sqrt{2\pi}} \int_{-\infty}^{\infty}z_{m,\ell} {\rm e}^{-\frac{(z_{m,\ell} -1)^2}{2\tau^2} - \frac{z_{m,\ell}^2}{2 \sigma^2}} \geq  0$. Therefore we calculate  
\begin{eqnarray*}
\E(w_{m,\ell} ) & = & \frac{1}{\sigma \sqrt{2\pi}} \int_{-\infty}^{\infty} {\rm e}^{-\frac{(z_{m,\ell} -1)^2}{2\tau^2} - \frac{z_{m,\ell}^2}{2 \sigma^2}} \\
   & = &\frac{{\rm e}^{-\frac{1}{2\tau^2}+ \frac{\sigma^2}{2(\sigma^2+\tau^2) \tau^2}}}{\sigma 2 \pi} \int_{- \infty}^{\infty} {\rm e} ^{ \frac{-\sigma^2+ \tau^2}{2\sigma^2\tau^2} ( z_{m,\ell}^2- \frac{2\sigma^2}{\sigma^2+ \tau^2}z_{m,\ell} + \frac{\sigma^4}{(\sigma^2+ \tau^2)^2})} \\
   & = & \frac{{\rm e}^{- \frac{1}{2(\sigma^2+\tau^2) }}}{\sigma} \sqrt{\frac{\sigma^2\tau^2}{\sigma^2+\tau^2}}= \frac{\tau {\rm e}^{- \frac{1}{2(\sigma^2+\tau^2) }}}{\sqrt{\sigma^2+\tau^2}}.
\end{eqnarray*}
This completes the proof.
\end{proof}

Define the $\Delta$-neighborhood of a pixel $(m,\ell)$ as $\mathcal{C}^{\Delta}_{m,\ell} = \{(i,j):  |i-m| \leq \Delta , |j-\ell| \leq \Delta \} \cap S$. 

\begin{lem}\label{lem:concsubgauss}
Let $\Omega_n = (2\Delta_n+1)^2$. We then have
\begin{align*}
\P\left(\frac{1}{\Omega_n} \left( \sum_{(m,\ell) \in \mathcal{C}^{\Delta_n}_{i,j} } w^{SY}_{m,\ell} - \sum_{(m,\ell) \in \mathcal{C}^{\Delta_n}_{i,j}} \E w^{SY}_{m,\ell} \right) \geq t \right)\leq2 {\rm e}^{-2\Omega_n t^2}. \\
\end{align*}
\end{lem}

The proof is a simple application of the Hoeffding inequality. 

\begin{proof}[Proof of Theorem \ref{thm:lowersusan}]
The first claim is that the optimal neighborhood size satisfies $\Delta_n = \Omega(\log n)$. We prove this by contradiction. Suppose that $\Delta_n = O(\log(n))$ and consider the performance
of the SYF on the image $x_{i,j} =0$ for every $(i,j)$.  It is clear that the bias is zero. However, the variance is lower bounded by $\Omega\left(\frac{1}{\log^2 n}\right)$. This is far from the optimal performance of the linear filters analyzed in Theorem  \ref{thm:linear}. Therefore $\Delta_n = \Omega(\log(n))$. 

Now consider the example image shown in Figure \ref{fig:horizedge} with $f_{h}(t_1,t_2) = \textbf{1}_{\{t_2< 0.5\}}$. For notational simplicity we assume $n$ is even so that the value of each pixel is either 0 or 1. Define the two regions $P_1 = \{(i,j) : \frac{n}{2} \leq j   \leq \frac{n}{2}+\frac{\Delta_n}{2} \}$ and $P_2 = \{(i,j) : j   >  \frac{n}{2}+{\Delta_n} \}$. At least $1/4$ of the pixels in the neighborhood of the pixels in $P_1$ have the noise-free value of 1. All pixels in the neighborhood of the pixels in $P_2$ have the noise-free pixel values equal to 1. Over each region we will find a lower bound for the risk of SYF and then sum them to obtain a lower bound for the risk over the entire image.

{\em Case I }-- $(i,j) \in P_1$: From the Jensen inequality we have
\begin{align*}
\E \left( x_{i,j}- \frac{\sum_{(m,\ell) \in \mathcal{C}^{\Delta_n}_{i,j}} w_{m,\ell}y_{m,\ell}}{\sum_{(m,\ell) \in \mathcal{C}^{\Delta_n}_{i,j}} w_{i,j}} \right)^2 \geq \left( \E \frac{\sum_{(m,\ell) \in \mathcal{C}^{\Delta_n}_{i,j}} w_{m,\ell}y_{m,\ell}}{\sum_{(m,\ell) \in \mathcal{C}^{\Delta_n}_{i,j}} w_{m,\ell}} \right)^2.
\end{align*}
Define the following two constants:
\begin{eqnarray*}
m_0 &=& \E(w_{i,j}^{SY}(m,\ell) \ | \ x_{i,j}= 0, x_{m,\ell} = 0), \\
m_1 &=& \E(w_{i,j}^{SY}(m,\ell)  \ | \ x_{i,j}= 0, x_{m,\ell} = 1). 
\end{eqnarray*}
It is clear that $m_0>m_1$. Let the event $A$ be
\begin{equation}\label{eq:eventA}
 A = \left\{\sum_{(m,\ell) \in \mathcal{C}^{\Delta_n}_{i,j}} w_{m,\ell} - \sum_{(m,\ell) \in \mathcal{C}^{\Delta_n}_{i,j}} \E w_{m,\ell} \leq \Delta_n^{2- \epsilon}  \right\}
\end{equation}
for some $\epsilon > 0$.  We have
\begin{eqnarray*}
\lefteqn{ \E \left( \frac{\sum_{(m,\ell) \in \mathcal{C}^{\Delta_n}_{i,j}} w_{m,\ell}y_{m,\ell}}{\sum_{(m,\ell) \in \mathcal{C}^{\Delta_n}_{i,j}} w_{m,\ell}} \right) \geq \E \left( \left. \frac{\sum_{(m,\ell) \in \mathcal{C}^{\Delta_n}_{i,j}} w_{m,\ell}y_{m,\ell}}{\sum_{(m,\ell) \in \mathcal{C}^{\Delta_n}_{i,j}} w_{m,\ell}}\  \right|  \ A \right) \P(A) } \\
&\overset{(a)}{\geq} & \! \! \! \! \E \! \left( \left. \frac{\sum_{(m,\ell) \in \mathcal{C}^{\Delta_n}_{i,j}} w_{m,\ell}y_{m,\ell}}{4\Delta_n^2 m_0 + \Delta_n^{2-\epsilon}}\  \right| \  A \right) \P(A) \! \hspace{5cm}\\
&\geq& \! \E \! \left( \frac{\sum_{(m,\ell) \in \mathcal{C}^{\Delta_n}_{i,j}} w_{m,\ell}y_{m,\ell}}{4\Delta_n^2 m_0 + \Delta_n^{2-\epsilon}} \right) -\P(A^c)  \\ 
&\overset{(b)}{\geq} & \! \! \!  \left( \frac{ \Delta_n^2 c_0}{4\Delta_n^2 m_0 + \Delta_n^{2-\epsilon}} \right) -P(A^c).
\end{eqnarray*}
Inequality $(a)$ uses Lemma \ref{lem:concsubgauss} and the fact that $m_0 \geq m_1$. Inequality $(b)$ uses Lemma \ref{lem:lowerbd_expec}, and therefore $c_0 = \frac{\tau}{\sqrt{\sigma^2 + \tau^2} {\rm e}^{-\frac{-1}{2(\sigma^2+ \tau^2)}}}$.  Since $\mathcal{C}^{\Delta_n}_{i,j}$ has $(2\Delta_n+1)^2$ pixels, at least $\Delta_n^2$ of them have the noise-free pixel value $1$. \\
Since $\Delta_n = \Omega(\log n)$, Lemma \ref{lem:concsubgauss} proves that $P(A^c)= o(1)$ and, therefore, the bias is lower bounded by $\Theta(1)$ for all of the pixels in $P_1$.

{\em Case II} -- $(i,j) \in P_2$:  As mentioned before, all pixels in the neighborhood of the pixels in P2 have the noise-free pixel values equal to 1. Hence, we have
\begin{align*}
\E\left(\frac{\sum_{(m,\ell) \in \mathcal{C}^{\Delta_n}_{i,j}} w_{m,\ell}y_{i,j}}{\sum_{(m,\ell) \in \mathcal{C}^{\Delta_n}_{i,j}} w_{m,\ell}}\right)^2 = \E\left(\frac{\sum_{(m,\ell) \in \mathcal{C}^{\Delta_n}_{i,j}} w_{m,\ell}z_{m,\ell}}{\sum_{(m,\ell) \in \mathcal{C}^{\Delta_n}_{i,j}} w_{m,\ell}}\right)^2.
\end{align*}
Defining the event $A$ as in \eqref{eq:eventA}, we have 
\begin{eqnarray*}
\lefteqn {\E\left( \left. \left(\frac{\sum_{(m,\ell) \in \mathcal{C}^{\Delta_n}_{i,j}} w_{m,\ell}z_{m,\ell}}{\sum_{(m,\ell) \in \mathcal{C}^{\Delta_n}_{i,j}} w_{m,\ell}} \right)^2 \ \right|  \ A\right) \P(A)} \\
&\geq&  \E\left( \left. \left( \frac{\sum_{(m,\ell) \in \mathcal{C}^{\Delta_n}_{i,j}} w_{m,\ell}z_{m,\ell}}{4\Delta_n^2 m_0+ \Delta_n^{2-\epsilon}} \right)^2\ \ \right| \ \ A\right) \P(A) \\
 &\geq & \E\left(\frac{\sum_{(m,\ell) \in \mathcal{C}^{\Delta_n}_{i,j}} w_{m,\ell}z_{m,\ell}}{4\Delta_n^2 m_0+ \Delta_n^{2-\epsilon}} \right)^2 - \P(A^c) = \frac{4 \Delta_n^2 \E (w_{m,\ell}z_{m,\ell})^2}{(4m_0\Delta_n^2+\Delta_n^{2-\epsilon})^2} - \P(A^c).
\end{eqnarray*}
If the neighborhood size is larger than $c \log(n)$ for some constant $c$, then Lemma \ref{lem:concsubgauss} will imply that $ \P(A^c) < o\left(\frac{1}{n^2}\right)$. Therefore, the dominant term in the above expression of the form of $\frac{\gamma}{\Delta_n^2}$. 
Combining the lower bounds for $P_1$ and $P_2$, we obtain a lower bound of the form of $\frac{\beta \Delta_n}{n}+ \frac{\gamma}{\Delta_n^2}$. Optimizing over $\Delta_n$ proves that
\[
\inf_{\Delta_n, \tau} R_n(f, \hat{f}^{SY}) > \Omega( n^{-2/3}).
\]  
This completes the proof.
\end{proof}

It is clear from the proof above that the neighborhood size is the main parameter that controls the decay rate of the risk of the YF. The Gaussian term in the YF weights enables an improvement in the constants but does not play any role in the decay rate. In the extreme case of $\Delta_n =n$, when all of the image pixels can potentially contribute to the estimation of a pixel, the decay rate of YF degrades to $\Theta(1)$. This algorithm is called the {\em range filter}, and \cite{ToMa98} observed in practice that it performs much worse than even linear filters, as the above analysis confirms. Interestingly, NLM addresses this issue and therefore its search space could be the entire image. This is the main reason for its improved performance.

The lower bound proved in Theorem \ref{thm:lowersusan} is the same as the upper bound we derived for the performance of linear filtering. Therefore, we have the following theorem. 

\begin{thm}
The risk of the SYF satisfies
\[
\inf_{\Delta_n, \tau}  \sup_{f \in {H}^{\alpha}(C)} R_n(f, \hat{f}^{SY}) \asymp n^{-2/3}.
\]
\end{thm}

\subsection{Proof of Theorem \ref{thm:upperbound}}\label{ssec:upperbound} 

The proof has two main steps. First, we show that the risk of the pixels far from the edge is $O( \log^{1+2\epsilon} (n)/n^2)$. Second, we show that the risk of the pixels whose $\delta_n$ neighborhood senses the edge is constant; however there are at most $O(n\delta_n)$ of these pixels. The following two lemmas will play key roles in our analysis.

\begin{lem}\label{lem:Gauss_exp_mean}
Let $Z \sim N(0,\sigma^2)$. For $\lambda < \frac{1}{2 \sigma^2}$, we have
\[
\E({\rm e}^{\lambda Z^2}) = \frac{1}{\sqrt{1-2\lambda \sigma^2}}.
\]
\end{lem}

\begin{proof}
The proof is a simple integral calculation:
\[
\E ({\rm e}^{\lambda Z^2})  = \frac{1}{\sigma \sqrt{2\pi}} \int_{-\infty}^{\infty} {\rm e}^{(\lambda - \frac{1}{2\sigma^2})Z^2 } dZ = \frac{1}{\sigma \sqrt{\frac{1}{\sigma^2} -2 \lambda }}.
\]
\end{proof}

\begin{lem}\label{thm:chisq_conc}
Let $Z_1, Z_2, \ldots, Z_n$ be iid $N(0,1)$ random variables. The $\chi^2_n$ random variable defined as $\sum_{i=1}^n Z_i^2$ concentrates around its mean with high probability, i.e.,
\begin{align*}
&\P\left(\frac{1}{n}\sum_i Z_i^2 -1 >t \right) \leq {\rm e}^{-\frac{n}{2}(t- \ln (1+t))}, \\
&\P\left( \frac{1}{n}\sum_i Z_i^2 -1 <-t \right) \leq {\rm e}^{-\frac{n}{2}(t+ \ln (1-t))}. \\
\end{align*}
\end{lem}

\begin{proof}
Here we prove just the first claim; the proof of the second claim follows along very similar lines. From Markov's Inequality, we have
\begin{eqnarray}\label{eq:prob:chisquare}
\lefteqn{\P\left (\left( \frac{1}{n} \sum_{i=1}^n Z_i^2 \right) -1 > t \right) \leq {\rm e}^{-\lambda t-\lambda} \E\left({\rm e}^{\frac{\lambda}{n} \sum_{i=1}^{n}  Z_i^2} \right)} \nonumber \\
&= &{\rm e}^{-\lambda t-\lambda} \left(\E\left({\rm e}^{\frac{\lambda Z_1^2}{n}} \right)\right)^n =\frac{ {\rm e}^{-\lambda t-\lambda}}{\left( 1- \frac{2\lambda}{n} \right)^{\frac{n}{2}}}.
\end{eqnarray}
The last inequality follows from Lemma \ref{lem:Gauss_exp_mean}. The upper bound proved above holds for any $\lambda < \frac{n}{2}$. To obtain the lowest upper bound we minimize $\frac{ {\rm e}^{-\lambda t-\lambda}}{\left( 1- \frac{2\lambda}{n} \right)^{\frac{n}{2}}}$ over $\lambda$. The optimal value of
$\lambda$ is  $\lambda^{\star} = \arg\min_{\lambda}  \frac{ {\rm e}^{-\lambda t-\lambda}}{\left( 1- \frac{2\lambda}{n} \right)^{\frac{n}{2}}} = \frac{nt}{2(t+1)}$. Plugging $\lambda^*$ into (\ref{eq:prob:chisquare}) proves the result.
\end{proof}

\begin{figure} 
\begin{center}
  \includegraphics[width=8cm]{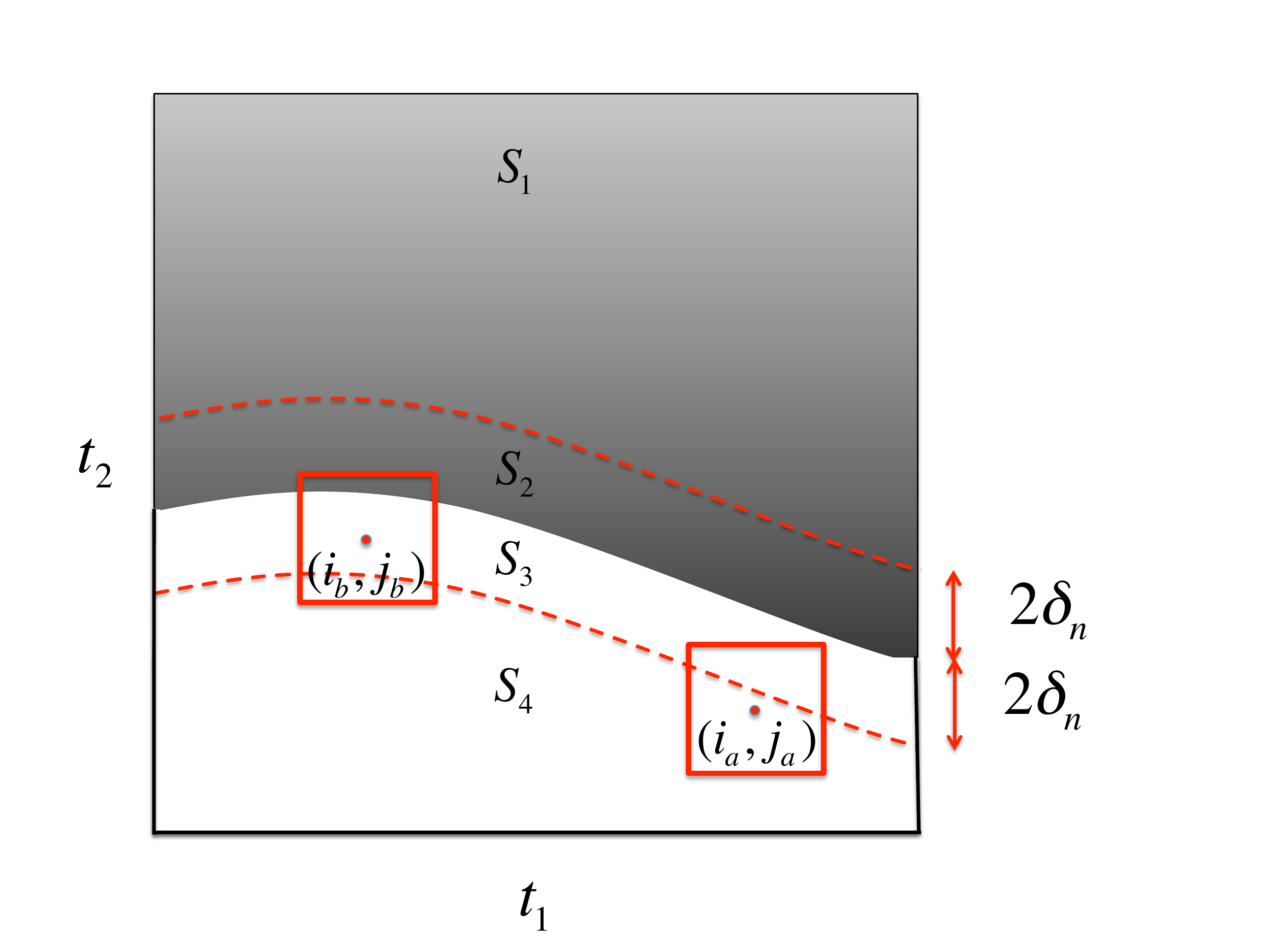}
  \caption{An example of a Horizon image. The $\delta_n$-neighborhood of pixel $(i_a,j_a) \in S_4$ does not intersect the edge contour, while the $\delta_n$-neighborhood of pixel $(i_b,j_b) \in S_3$ intersects with the edge contour.}
  \label{fig:regions}
  \end{center}
\end{figure}

\begin{proof}[Proof of Theorem \ref{thm:upperbound}] We will consider the following partition of the image pixels. Let $S=\{1,2, \ldots, n \} \times \{1,2,\ldots,n \}$. For a given Horizon function $f_h(t_1,t_2)$, define $S_1 = \{(i,j) \ | \  \frac{j}{n} > h(\frac{i}{n}) + \frac{{2}\delta_n}{n} \}$, $S_2 =  \{(i,j) \ | \    h(\frac{i}{n}) < \frac{j}{n} \leq h(\frac{i}{n}) + \frac{{2}\delta_n}{n} \}$, $S_3 =  \{(i,j) \ | \    h(\frac{i}{n})- \frac{{2}\delta_n}{n}  \leq \frac{j}{n} \leq h(\frac{i}{n})  \}$, and $S_4 =  \{(i,j) \ | \    \frac{j}{n} < h(\frac{i}{n})-\frac{{2} \delta_n}{n}  \}$. These regions are displayed in Figure \ref{fig:regions}.The $\delta_n$-neighborhood of the pixels in $S_1$ and $S_4$ do not intersect the edge, while the $\delta_n$-neighborhood of the other pixels may have pixels from both sides of the edge. See Figure \ref{fig:regions}. For the notational simplicity we write $\sum_{(i,j) \in S_{\ell}}$ for the double summation over $i,j$ where $j$ satisfies the constraints specified for $S_{\ell}$.

Consider a pixel $(i,j) \in S_1$. The risk of NLM at this pixel is
\[
\E \left(x_{i,j} - \frac{\sum w_{m,\ell} y_{m,\ell}    }{ \sum w_{m,\ell}}\right)^2,
\] 
where $x_{i,j}=0$, since $(i,j) \in S_1$. Define the set of oracle weights
\begin{eqnarray}\label{eqn:oracleweights}
w^{\star}_{m,\ell} =\! \! \left\{\begin{array}{rl}
1 &  \mbox{ if $\frac{\ell}{n} > h(\frac{m}{n})$,}\\
 0 &   \mbox{ otherwise.}
\end{array}\right.
\end{eqnarray}
Define $U \triangleq \left( \frac{\sum w_{m,\ell} y_{m,\ell}    }{ \sum w_{m,\ell}} \right)^2$, and let the event $A = \{ w_{m,\ell} = w^{\star}_{m,\ell}$, $\forall (m, \ell) \in S_1\cup  S_4\} $. We then have
\begin{align}\label{eq:nlmrisk:main}
\E (U ) &= \E(U\ | \  A)\P(A) + \E (U \ | \ A^c ) \P(A^c)  \leq \E(U\ | \  A)\P(A) +  \P(A^c),
\end{align}
where the last inequality is due to the fact that the risk of the estimator is bounded by 1.
We now calculate each term of \eqref{eq:nlmrisk:main} separately.

Define $S_{14}=S_1 \cup S_4$ and $S_{23}=S_2 \cup S_3$.  Then we have
\begin{eqnarray} \label{eq:thmupp1}
 \lefteqn{\E (U\ | \  A)\P(A)} \nonumber \\  
 &=&  \E \left(  \left. \left( \frac{\sum_{(m,\ell) \in S_{14}} w^{\star}_{m,\ell} y_{m,\ell} +    \sum_{(m,\ell) \in S_{23}} w_{m,\ell} y_{m,\ell}  }{ \sum_{(m,\ell) \in S_{14}} w^{\star}_{m,\ell} + \sum_{(m,\ell) \in S_{23}} w_{m,\ell}}\right)^2 \  \right| \  A\right) \P(A)  \nonumber \\
&\leq& \E   \left( \frac{\sum_{(m,\ell) \in S_{14}} w^{\star}_{m,\ell} y_{m,\ell} +    \sum_{(m,\ell) \in S_{23}} w_{m,\ell} y_{m,\ell}  }{ \sum_{(m,\ell) \in S_{14}} w^{\star}_{m,\ell} + \sum_{(m,\ell) \in S_{23}} w_{m,\ell}}\right)^2 \nonumber \\
 &\leq & \E  \left( \frac{\sum_{(m,\ell) \in S_{14}} w^{\star}_{m,\ell} x_{m,\ell} +    \sum_{(m,\ell) \in S_{23}} w_{m,\ell} x_{m,\ell}  }{ \sum_{(m,\ell) \in S_{14}} w^{\star}_{m, \ell} + \sum_{(m,\ell) \in S_{23}}w_{m,\ell}  }\right)^2 \nonumber \\
  & & + ~ \E   \left( \frac{\sum_{(m,\ell) \in S_{14}} w^{\star}_{m,\ell} z_{m,\ell} +    \sum_{(m,\ell) \in S_{23}} w_{m,\ell}  z_{m,\ell}  }{ \sum_{(m,\ell) \in S_{14}} w^{\star}_{m, \ell} + \sum_{(m,\ell) \in S_{23}} w_{m,\ell} }\right)^2  \nonumber \\
 & & + ~ 2 \sqrt{ \E  \left(\frac{\sum_{(m,\ell) \in S_{14}} w^{\star}_{m,\ell} x_{m,\ell} +    \sum_{(m,\ell) \in S_{23}} w_{m,\ell} x_{m,\ell}  }{ \sum_{(m,\ell) \in S_{14}} w^{\star}_{m,\ell} + \sum_{(m,\ell) \in S_{23}}w_{m,\ell}  }\right)^2} \nonumber \\
& & \times ~\sqrt{ \E   \left( \frac{\sum_{(m,\ell) \in S_1 \cup S_4} w^{\star}_{m,\ell} z_{m,\ell} +    \sum_{(m,\ell) \in S_{23}} w_{m,\ell}  z_{m,\ell}  }{ \sum_{(m,\ell) \in S_{14}} w^{\star}_{m,\ell} + \sum_{(m,\ell) \in S_{23}} w_{m,\ell} }\right)^2}.
\end{eqnarray}
The last inequality is due to Cauchy-Schwartz. In the next two lemmas we bound the last three terms of \eqref{eq:thmupp1}. 
\begin{lem} \label{lem:upper1}
Let $w_{m,\ell}$ be the weights of NLM with $\delta_n = \log^{\frac{1}{2}+\epsilon} n$ and $t_n = \frac{2}{\sqrt{\log^{\frac{\epsilon}{2}} n}}$ for $\epsilon>0$. Also, let $w^{\star}_{m,\ell}$ be the oracle weights introduced in (\ref{eqn:oracleweights}). Then
\begin{align*}
\E   \left( \frac{\sum_{(m,\ell) \in S_{14}} w^{\star}_{m,\ell} x_{m,\ell} +    \sum_{(m,\ell) \in S_{23}} w_{m,\ell} x_{m, \ell}  }{ \sum_{(m,\ell) \in S_{14}} w^{\star}_{m,\ell} + \sum_{(m,\ell) \in S_{23}}w_{m,\ell}  }\right)^2 = O\left(\frac{\delta^2_n}{n^2}\right).
\end{align*}
\end{lem}

\begin{proof}
Define $S_f$ as the set of indices of the pixels whose noise-free value is neither zero nor one. Since the images are chosen from the Horizon class, the cardinality of this set is at most $2n$. Plugging in the values of $x_{m,\ell}$, we have
\begin{eqnarray}
 \lefteqn{\E  \left( \frac{\sum_{(m,\ell) \in S_{14}} w^{\star}_{m,\ell} x_{m,\ell} +    \sum_{(m,\ell) \in S_{23}} w_{m\ell} x_{m,\ell}  }{ \sum_{(m,\ell) \in S_{14}} w^{\star}_{m, \ell} + \sum_{(m,\ell) \in S_{23}}w_{m,\ell}  }\right)^2} \nonumber \\
&\overset{(a)}{=} &\! \! \E \left( \frac{\sum_{(m,\ell) \in S_{14}} w^{\star}_{m,\ell} x_{m,\ell} +  \!  \sum_{(m,\ell) \in S_{3}\backslash S_f} w_{m,\ell} + \! \sum_{(m,\ell) \in S_f} w_{m,\ell}x_{m,\ell}  }{ \sum_{(m,\ell) \in S_{14}} \! w^{\star}_{m,\ell} + \! \sum_{(m,\ell) \in S_{3} \backslash S_f} \! w_{m,\ell} + \! \sum_{(m,\ell) \in S_{2} \backslash S_f} \! w_{m,\ell}  + \! \sum_{(m,\ell) \in S_f} w_{m\ell} }\right)^2\nonumber \\
 &\overset{(b)}{\leq} & \! \! \E  \left( \frac{\sum_{(m,\ell) \in S_{14}} w^{\star}_{m,\ell} x_{m,\ell} +    \sum_{(m,\ell) \in S_{3}} 1 + \sum_{(m,\ell) \in S_f} w_{m\ell}x_{m,\ell} }{ \sum_{(m,\ell) \in S_{14}} w^{\star}_{m,\ell} + \sum_{(m,\ell) \in S_{3}}1 + \sum_{(m,\ell) \in S_f} w_{m,\ell}}\right)^2 \nonumber \\
 &{\leq} &\! \! \E  \left( \frac{\sum_{(m,\ell) \in S_{14}} w^{\star}_{m,\ell} x_{m,\ell} +    \sum_{(m,\ell) \in S_{3}} 1 + 2n }{ \sum_{(m,\ell) \in S_{14}} w^{\star}_{m,\ell} + \sum_{(m,\ell) \in S_{3}}1 }\right)^2 = O\left(\frac{\delta^2_n}{n^2}\right), \nonumber
\end{eqnarray}
where Inequality (b) is due to the fact that the expression after Equality (a) is an increasing function of $ \sum_{(m,\ell) \in S_{3}\backslash S_f}w_{m\ell}$ and a decreasing function of $\sum_{(m,\ell) \in S_{2}\backslash S_f}w_{m\ell}$. Therefore, we set $w_{m,\ell}=1$ for $(m,\ell) \in S_3$ and $w_{m,\ell}=0$ for $(m,\ell) \in S_2$.
\end{proof}

\begin{lem}\label{lem:upper2}
Let $w_{m,\ell}$ be the weights of NLM with $\delta_n = \log^{\frac{1}{2}+\epsilon} n$ and $t_n = \frac{2}{\sqrt{\log^{\frac{\epsilon}{2}} n}}$ for $\epsilon>0$. Also, let $w^{\star}_{m,\ell}$ be the oracle weights introduced in (\ref{eqn:oracleweights}). Then we have
\begin{align*}
 \E   \left( \frac{\sum_{(m,\ell) \in S_{14}} w^{\star}_{m,\ell} z_{m,\ell} +    \sum_{(m,\ell) \in S_{23}} w_{m,\ell}  z_{m,\ell}  }{ \sum_{(m,\ell) \in S_{14}} w^{\star}_{m,\ell} + \sum_{(m,\ell) \in S_{23}} w_{m,\ell} }\right)^2 = O\left(\frac{1}{n^2}\right).
\end{align*}
\end{lem}

\begin{proof}
Since $\sum_{(m,\ell) \in S_{23}} w_{m,\ell} \geq 0$ and we are interested in the upper bound of the risk, we can remove it from the denominator to obtain
\begin{eqnarray}\label{eq:nlmvariance}
\lefteqn{ \E \left(  \left( \frac{\sum_{(m,\ell) \in S_{14}} w^{\star}_{m,\ell} z_{m,\ell} +    \sum_{(m,\ell) \in S_{23}} w_{m,\ell}  z_{m,\ell}  }{ \sum_{(m,\ell) \in S_{14}} w^{\star}_{m,\ell} + \sum_{(m,\ell) \in S_{23}} w_{m,\ell} }\right)^2\right)}  \nonumber \\
&{\leq} &  \E \left(  \left( \frac{\sum_{(m,\ell) \in S_{14}} w^{\star}_{m,\ell} z_{m,\ell} +    \sum_{(m,\ell) \in S_{23}} w_{m,\ell}  z_{m,\ell}  }{ \sum_{(m,\ell) \in S_{14}} w^{\star}_{m,\ell} }\right)^2\right)  \nonumber \\
&= &  \E \left(  \left( \frac{\sum_{(m,\ell) \in S_{14}} w^{\star}_{m,\ell} z_{m,\ell} }{ \sum_{(m,\ell) \in S_{14}} w^{\star}_{m,\ell} } \right)^2 \right) +
  \E \left(  \left( \frac{ \sum_{(m,\ell) \in S_{23}} w_{m,\ell}  z_{m,\ell}  }{ \sum_{(m,\ell) \in S_{14}} w^{\star}_{m,\ell} }\right)^2\right)  \nonumber \\
& & + ~ 2 \E \left(  \left( \frac{\sum_{(m,\ell) \in S_{14}} w^{\star}_{m,\ell} z_{m,\ell} }{ \sum_{(m,\ell) \in S_{14}} w^{\star}_{m,\ell} }\right) \left(  \frac{ \sum_{(m,\ell) \in S_{23}} w_{m,\ell}  z_{m,\ell}  }{ \sum_{(m,\ell) \in S_{14}} w^{\star}_{m,\ell} }\right)\right).
\end{eqnarray}
Since $ \frac{\sum_{(m,\ell) \in S_{14}} w^{\star}_{m,\ell} z_{m,\ell} }{ \sum_{(m,\ell) \in S_{14}} w^{\star}_{m,\ell} }$ is the average of iid random variables, it is not hard to prove that $\E  \left( \frac{\sum_{(m,\ell) \in S_{14}} w^{\star}_{m,\ell} z_{m,\ell} }{ \sum_{(m,\ell) \in S_{14}} w^{\star}_{m,\ell} } \right)^2 = O(\frac{\sigma^2}{n^2})$. To bound the other two terms in \eqref{eq:nlmvariance} we use the notation defined in the last section: $\mathcal{C}^{\Delta}_{m,\ell} = \{(i,j):  |i-m|<\Delta , |j-\ell| < \Delta \} \cap S$. We also define $\E(\cdot \ | \ \mathcal{C}^{\Delta}_{m, \ell}  )$ as the conditional expectation given the variables in $\mathcal{C}^{\Delta}_{m,\ell} $. We then have  
\begin{eqnarray*}
 \lefteqn {\E \left(  \left( \frac{ \sum_{(m,\ell) \in S_{23}} w_{m,\ell}  z_{m,\ell}  }{ \sum_{(m,\ell) \in S_{14}} w^{\star}_{m,\ell} } \right)^2 \right) } \\
 &=&  \E \left(  \E \left( \left. \left( \frac{ \sum_{(m,\ell) \in S_{23}} w_{m,\ell}  z_{m,\ell}  }{ \sum_{(m,\ell) \in S_{14}} w^{\star}_{m,\ell} } \right)^2 \ \right| \ \mathcal{C}^{\delta_n}_{i,j} \right )\right)  \\
 &= &  \E \left(     \frac{ \E (\sum_{(m',\ell')  \in S_{23}} \sum_{(m,\ell) \in S_{23}} w_{m,\ell}  z_{m,\ell}  w_{m',\ell'}  z_{m',\ell'} \  | \ \mathcal{C}^{\delta_n}_{i,j} )}{ (\sum_{(m,\ell) \in S_{14}} w^{\star}_{m,\ell})^2 } \right) \\
 &= & \E \left(     \frac{ \E (\sum_{(m',\ell')  \in \mathcal{C}^{2\delta_n}_{m,\ell} } \sum_{(m,\ell) \in S_{23}} w_{m,\ell}  z_{m,\ell}  w_{m',\ell'}  z_{m',\ell'} \  | \ \mathcal{C}^{\delta_n}_{i,j} )}{ (\sum_{(m,\ell) \in S_{14}} w^{\star}_{m,\ell})^2 } \right)  \\
 &= &  \left(  \frac{  \sum_{(m',\ell')  \in \mathcal{C}^{2\delta_n}_{m,\ell} } \sum_{(m,\ell) \in S_{23}} \E( w_{m,\ell}  z_{m,\ell}  w_{m', \ell'}  z_{m', \ell'} )}{ (\sum_{(m,\ell) \in S_{1 4}} w^{\star}_{m,\ell})^2 } \right)  \leq 
 O\left(\frac{\delta_n^3 }{n^3}\right).
\end{eqnarray*} 
For the last inequality we have used the Cauchy-Schwartz Inequality to prove that $ \E( w_{m,\ell}  z_{m,\ell}  w_{m',\ell'}  z_{m',\ell'} ) \leq 3 \sigma^2$. Although we could derive a loose
bound for $\E \left(  \left( \frac{ \sum_{(m,\ell) \in S_{23}} w_{m,\ell}  z_{m,\ell}  }{ \sum_{(m,\ell) \in S_{14}} w^{\star}_{m,\ell} } \right)^2 \right)$ and still draw the same conclusion, we used the above technique since we have to use it in the proof of Theorem \ref{thm:taperedweights}. The last term we have to bound in \eqref{eq:nlmvariance} is 
\begin{eqnarray*}
\lefteqn{\E \left(  \left( \frac{\sum_{(m,\ell) \in S_{14}} w^{\star}_{m,\ell} z_{m,\ell} }{ \sum_{(m,\ell) \in S_{14}} w^{\star}_{m,\ell} }\right)  \left( \frac{ \sum_{(m,\ell) \in S_{23}} w_{m,\ell}  z_{m,\ell}  }{ \sum_{(m,\ell) \in S_{14}} w^{\star}_{m,\ell} } \right) \right)}\\
&  \leq &\sqrt{\E \left( \frac{\sum_{(m,\ell) \in S_{14}} w^{\star}_{m,\ell} z_{m,\ell} }{ \sum_{(m,\ell) \in S_{14}} w^{\star}_{m,\ell} }\right)^2 } \sqrt{ \E \left( \frac{ \sum_{(m,\ell) \in S_{23}} w_{m,\ell}  z_{m,\ell}  }{ \sum_{(m,\ell) \in S_{14}} w^{\star}_{m,\ell} } \right)^2} \\
&\leq&  O\left(\frac{1 }{n^2}\right).
\end{eqnarray*}
This proves the lemma.
\end{proof}

Using Lemma \ref{lem:upper1} and Lemma \ref{lem:upper2} in (\ref{eq:thmupp1}) proves that
\begin{align}\label{eq:upperboundforU}
\E(U \, |  \, A) \P(A) = O \left(\frac{\delta_n^{2}}{n^2}\right).
\end{align}  
Finally, using Lemma \ref{thm:chisq_conc} and the union bound it is easy to show that
\begin{align}\label{eq:upperboundAc}
  \P(A^{c}) = O\left( \frac{1}{n^2}\right).
\end{align}
  It is important to note that the constants of this probability are hidden in the $O$ notation. These constants depend on $\epsilon$ and
  increase as $\epsilon$ decreases. Therefore, we cannot set $\epsilon =0$. 
  
% \begin{proof}
%  First, we calculate $\P(w_{m,\ell} \neq w^{\star}_{m,\ell} \ | \ (m,\ell) \in S_1)$.
%  
%  \begin{align*}
%&\P(w_{m,\ell} \neq w^{\star}_{m,\ell} \ | \ w^{\star}_{m,\ell} =1) = \P(d_{\delta_n}^2(y_{m,\ell}, y_{i,j}) \geq 2\sigma^2+ t_n) = \\
%&\P(\frac{1}{(2\delta_n+1)^2}\sum_n \sum_p |z_{i+n,j+p}-z_{m+n,\ell+p}|^2 \geq 2\sigma^2+ t_n))
%\end{align*}
%If $s_{np} = z_{i+n,j+p}-z_{m+n,\ell+p}$ then $s_{np} \sim N(0, 2\sigma^2)$. To simplify the notation we assume that we are considering two non-overlapping neighborhoods. (In case we have overlapping neighborhoods, it is easy to see that we can break the summation to two sets of independent random variables 
%and write similar concentration of measure for each. Since we will have similar calculations for the lower bound, here we just explain the simpler case of non-overlapping neighborhoods) Under this assumption $s_{n,p}$ random variables are independent. According to Theorem \ref{thm:chisq_conc} proved in the appendix 
%we can bound the probability. 
%  \begin{align*}
% & \P(\frac{1}{(2\delta_n+1)^2}\sum_n \sum_p |z_{i+n,j+p}-z_{m+n,\ell+p}|^2 \geq 2\sigma^2+ t_n))  \\
%  &\leq {\rm e}^{2 \log^2 n(t_n- \ln(1+t_n))}  \leq O({\rm e} ^{-4\log n}) = O(\frac{1}{n^4}).
%  \end{align*}
% 
% Very similarly using Theorem  \ref{thm:chisq_conc}, we can prove that $\P(w_{m,\ell} \neq w^{\star}_{m,\ell} \ | \ (m,\ell) \in S_1) = O(\frac{1}{n^4})$. 
% Finally using the union bound we have,
% \[
% \P(A^c)  \leq \sum_{(m,\ell) \in S_1 \cup S_4} \P(w_{m,\ell} \neq w^{\star}_{m,\ell}) =O(\frac{1}{n^2}).
% \]
% \end{proof}
  Plugging in \eqref{eq:upperboundAc} and \eqref{eq:upperboundforU} in \eqref{eq:nlmrisk:main} results in
  \[
\E \left( x_{i,j} - \frac{\sum w_{m,\ell} y_{m,\ell}    }{ \sum w_{m,\ell}}\right)^2 = O\left(\frac{\log^{1+2\epsilon} (n)}{n^2}\right)  \ \ \ \ \  \forall (i,j) \in S_1.
\] 
 Now consider $(i,j) \in S_2 \cup S_3$. In this region we can bound the error by the worst possible risk, which is $1$. We will discuss the sharpness of this bound in the next section where we develop a lower bound for the risk. 
 
Using the bounds provided above for the risks of the pixels in $S_1, S_2, S_3$ and $S_4$, we can now calculate the final upper bound for the risk of the NLM as
\begin{eqnarray*}
\sup_{f \in H^{\alpha}(C)}  R(f, \hat{f}^{NL}) &=& \frac{1}{n^2} \sum_i \sum_j \E(x_{i,j} - \hat{f}^{N}_{i,j})^2 \\
 &\leq& \frac{\log^{1+2\epsilon} (n) (|S_1| +|S_4|)}{n^2} + \frac{|S_2|+ |S_3|}{n^2} \\
 & \leq&  O\left(\frac{\log^{\frac{1}{2}+ \epsilon }(n)}{n}\right).
\end{eqnarray*}
In order to derive the last inequality we noted that since $h(t_1) \in H\ddot{o}lder^{1}(1)$ the cardinality of $S_2$ and $S_3$ are $O(n\log(n))$. This completes the proof of Theorem \ref{thm:upperbound}.
\end{proof}

%--------------------------------------------
% Lower bound
%---------------------------------------------

\subsection{Proof of Theorem \ref{thm:lowerbound}}\label{ssec:lowerbound}

Suppose that the parameters of SNLM satisfy assumptions A1--A4.  To derive a lower bound we consider the performance of the SNLM algorithm on the simple image in Figure \ref{fig:horizedge}. For notational simplicity we assume that $n$ is even, and hence all of the pixel values are either $0$ or $1$. The proof follows four main steps: 
 \begin{enumerate}
 
 \item We consider the pixels that are just above the edge, i.e., $(i, \lceil \frac{n}{2} \rceil)$, and prove that the risk of the NLM on these pixels is lower bounded by a constant that does not depend on $n$. 
 
 \item Using asymptotic arguments we prove that the probability a pixel just below the edge passes the threshold $t_n >0$ is larger than $p_0$, where $p_0$ is a non-zero probability independent of $n$. Based on this, we use a concentration argument to prove that $\Theta(n)$ of the pixels just below the edge will pass the threshold with high probability. See the formal statement in Theorem \ref{thm:belowedge}. 
 
 \item Using symmetry arguments we prove that the probability a pixel that is $\ell< \delta_n/2$ rows\footnote{The $\ell^{th}$ \textit{row} of an image is the set of all pixels of the form $(i,\ell)$.} above the edge or below the edge passes the threshold is equal. This is formally stated in Lemma \ref{lem:colsymmetry}. 
 
 \item Combining the outcomes of Steps 2 and 3 we show that the risk is minimized if all the pixels just above the edge pass the threshold and the probability that the other pixels pass the threshold is as low as possible. If more zero pixels above the edge pass the threshold, then more pixels with noise-free value $1$ will also pass the threshold, and this makes the bias large. Therefore we assume that $p_{n,\ell}$, the probability that a pixel at distance $\ell$ of the edge passes the threshold, is equal to zero for $\ell>1$. However, we have already proven that for $\ell =1$ the probability is larger than $p_0$. Theorem \ref{thm:lowerbound} uses this fact to show that the risk of this estimator is larger than a constant independent of $n$.
 
 \end{enumerate}
 
\begin{prop} \label{thm:belowedge}
Let $j^*= \lceil \frac{n}{2} \rceil$. For any pixel with coordinates of the form $(i^*,j^*)$, there exists a
non-zero constant probability $p_0$ such that for any $\delta_n$ and $t_n$
\[
\P \left( \sum_m {w_{m,j^*-1} }-np_0   < -t\ \right) \leq 4 \delta_n {\rm e}^{-\frac{t^2}{4n \delta_n}}.
\] 
%i.e. with very high probability $\Theta(n)$ of the pixels just below the edge will contribute in the
%nonlocal averaging. 
\end{prop}

\begin{proof}
For notational simplicity we use $i = i^*$ and $j=j^*$ in the proof.   We have
\begin{eqnarray*}
\lefteqn {\P (\bar{d}^2_{\delta_n}(y_{i,j}, y_{m,j-1})   \leq \sigma^2 + t_n  )}  \\
&=&\! \! \! \! \P \left( \frac{1}{\rho_n^2}( \sum_{\ell,p} | x_{i+p, j+\ell}- y_{m+p,j-1+\ell}|^2
 - ( x_{i,j}- y_{p,j-1} ) ^2)\leq  \sigma^2+ t_n \right)  \\
&=& \! \! \! \!  \P \left(\frac{1}{\rho_n^2} \sum_{\ell,p} (s^2_{\ell,p}-\sigma^2) -\frac{2}{\rho_n^2} \sum_{\ell} s_{\ell,0} \leq -\frac{1}{\rho_n} + t_n \right)  \\
& \geq & \! \! \! \! \P \left(\frac{1}{\rho_n^2} \sum_{\ell,p} (s^2_{\ell,p}-\sigma^2) -\frac{2}{\rho_n^2} \sum_{\ell} s_{\ell,0} \leq -\frac{1}{\rho_n} \right),
\end{eqnarray*}
where $s_{\ell,m} = z_{m+\ell,j-1+p}$. According to the Berry-Esseen Central Limit Theorem for independent non-identically distributed random variables \cite{Stein86}, we know that
\begin{align*}
&\P \left( \frac{1}{\rho_n^2} \sum_{\ell} \sum_p (s^2_{\ell,p}-\sigma^2) -\frac{2}{\rho_n^2} \sum_{\ell} s_{\ell,0} \leq -\frac{1}{\rho_n}  \right) \geq 
\P(G \leq -1 ) - \frac{C}{\rho_n},
\end{align*}
where $G$ is a Gaussian random variable with mean zero and bounded standard deviation. In fact, it is not difficult to confirm that
\[
\E(G^2) = 2 \sigma^4 + \frac{8 \sigma^2 \delta_n - 2\sigma^4}{(2 \delta_n+1)^2}. 
\]
 Since $\P(G \leq -1 ) \geq 2p_0$ ($2p_0$ is $P(G' \leq -1)$ where $G' \sim N(0, 2\sigma^4)$) is non-zero, for large values of $n$ we can 
ensure that $C/n <p_0$ and therefore that $\P ( \bar{d}^2_{\delta_n}(y_{i,j}, y_{m,j-1})   \leq  \sigma^2 + t_n  )  > p_0$. We now prove that
even though the weights are correlated, $\Theta(n)$ of the weights will be equal to $1$ with very high probability. Define $u_i$ as $w_{i, j-1}$ and define the process $U = (u_1, \ldots, u_n)$. 
Break this sequence into $2\delta_n$ subsequences $U_i= (u_{i}, u_{i+2\delta_n}, u_{i+4\delta_n}, \ldots, u_{n-2\delta_n+i} )$.  Each $U_i$ has independent and identically distributed elements. Therefore, according to the Hoeffding Inequality, we have $\P( | \sum_{u_j \in U_i} u_j - \frac{n}{2 \delta_n} \E(u_i)|  > t ) \leq 2 {\rm e} ^{\frac{-t^2 \delta_n}{n}}$. On the other hand we know that $E(u_i)>p_0$. Therefore,
\[
\P \left( \sum_{u_j \in U_i} u_j < \frac{n}{2\delta_n}p_0- t \right) \leq 2 {\rm e}^{\frac{-t^2 \delta_n}{n}}.
\]
Finally we use the union bound to obtain 
\begin{eqnarray*}
\lefteqn {\P \left( \sum u_i - np_0  \leq -t \right) \leq \P\left(\sum_i \sum_{u_j \in U_1} u_j - \frac{n}{2\delta_n} p_0  \leq -t \right)}  \\
&\leq  &\P \left( \cup_i \{\omega:  \sum_{u_j \in U_i} u_j - \frac{n}{2\delta_n} p_0  \leq -\frac{t}{2\delta_n} \} \right) \leq 4 \delta_n {\rm e } ^{-\frac{t^2}{4n\delta_n}}.
 \end{eqnarray*}
\end{proof}

Define the set $J = \{(i,j) \ | \ j= \lfloor j h(\frac{i}{n}) \rfloor \}$. It is clear that $|J| = n$. The following Corrollary to Proposition \ref{thm:belowedge} shows that NLM sets the weights of most of the pixels in $J$ to $1$.

\begin{cor}
Consider the image displayed in Figure \ref{fig:horizedge}, and let $\delta_n = O(n^{\alpha})$ for $\alpha <1$. For any $\delta_n$ and $t_n>0$, $\Theta(n)$ of the pixels in $J$ will pass the threshold $t_n$ with very high probability. 
\end{cor}

\begin{proof}
Set $t = n^{\frac{3+ \alpha}{4}}$ in Proposition \ref{thm:belowedge}. 
\end{proof}

Remarkably the above corollary holds in a very general setting even if the assumptions A1--A4 do not hold. In other words, NLM in its most general form is not able to distinguish between the pixels right above the edge from the pixels right below the edge. This is due to the fact that the ``signal to noise ratio" in the $\delta_n$-neighborhood distance estimates is too low at the edge pixels. This is the result of the isotropic neighborhoods used to form the weight estimates. 

%\begin{cor}
%Suppose that $\delta_n = \log (n)$ and with set $k_n$ equal to $\frac{1}{2}\delta_n$. This is the maximum value one can imagine for $k_n$. Then the performance of the NL means algorithm is
%lower bounded by,
%\[
%\Theta(\frac{1}{n\log^2 n})
%\]
%\end{cor}

%This upto the the $\log(n)$ factor from the upper bound that we have proved.

%\begin{cor}
%Suppose that $k_n$ does not depend on $n$ and is constant in terms of $n$. Then the lower bound provided here is going to be,
%\[
%\Theta(\frac{1}{n}), 
%\]
%which is again a log factor away from the upper bound provided here. 
%\end{cor}

\begin{lem}\label{lem:rowsym}
 If $|m-i^*| > \delta_n/2$ and $|m'-i^*| > \delta_n/2$, then
\begin{align*}
\P (\bar{d}^2_{\delta_n} (y_{i^*,j^*}, y_{m, j^*-\ell}) \leq \sigma^2+ t_n) =  
\P (\bar{d}^2_{\delta_n} (y_{i^*,j^*}, y_{m', j^*-\ell}) \leq \sigma^2+ t_n)
\end{align*}
for any $\ell,m,m'$.
\end{lem}

The proof of this lemma is obvious and is skipped here.

\begin{lem}\label{lem:colsymmetry}
For $\ell< \delta_n/2$,
\begin{align*}
\P (\bar{d}^2_{\delta_n} (y_{i^*,j^*}, y_{m, j^*-\ell}) \leq \sigma^2+ t_n) =  
\P (\bar{d}^2_{\delta_n} (y_{i^*,j^*}, y_{m, j^*+\ell}) \leq \sigma^2+ t_n).
\end{align*}
\end{lem}

The proof of this lemma is also obvious from symmetry and is skipped here. We can now prove Theorem \ref{thm:lowerbound}, which provides a lower bound for the risk of SNLM.

\begin{proof}[Proof of Theorem \ref{thm:lowerbound}]
We derive a lower bound for the risk of SNLM on the image displayed in Figure \ref{fig:horizedge}. To do so, we consider the pixels just above the edge and prove that the SNLM algorithm has risk $\Theta(1)$ at these pixels. Since there are $\Theta(n)$ of these pixels, the risk over the entire image is larger than $\Theta(n^{-1})$.
 
Consider a pixel $(i^*,j^*)$ with $j^* = \lceil \frac{n}{2} \rceil$. The risk of the SNLM is 
\begin{align}\label{eq:nlmlower:bias}
\E\left(f_{i^*,j^*}- \frac{\sum \sum w_{m ,\ell} y_{m,\ell}}{\sum \sum w_{m,l}}\right)^2  \geq \left (\E \left( \frac{\sum \sum w_{m ,\ell} y_{m,\ell}}{\sum \sum w_{m,l}}\right) \right)^2.
\end{align}
Note that $w_{m,\ell}$ is independent of the $y_{m,\ell}$ according to the construction of the SNLM weights in \eqref{eq:$NL_1$}. Let $p_{n,\ell}$ be the probability $\P(w_{\ell, i}=1)$ for $\ell \in \{j^*-\delta_n, j^*-\delta_n+1, \ldots, j^*+\delta_n \}$. We can partition the row $\{(i,\ell ) \, | \,1 \leq i \leq n \}$ into $2\delta_n+1$ subsequences and apply Hoeffding inequality on each subsequence. We combine the results of different subsequences with the union bound to prove that
\begin{equation}\label{eq:concentration}
\P\left(|\sum_m w_{m,\ell}- np_{n,\ell}|> t  \right) \leq 4\delta_n {\rm e}^{\frac{-t^2}{4n \delta_n}}.
\end{equation}
 Define the event $A$ as 
 \[
 A = \left\{  |\sum_m w_{m,\ell}- np_{n,\ell}|< n^{0.66} \  \   \forall    \ell, |\ell- j^*| \leq \delta_n  \right\}.
\]
Using the union bound and (\ref{eq:concentration}) we have
\[
P(A^c) \leq 8\delta_n^2  {\rm e}^{\frac{-n^{1.32}}{4n \delta_n}}.
\]
Any lower bound on the bias of the estimator leads to a lower bound on its risk. Therefore, we find a lower bound for the bias as follows:
\begin{eqnarray}\label{eq:indconverted}
 \lefteqn {\E \left( \left. \frac{\sum \sum w_{m ,\ell} y_{m,\ell}}{\sum \sum w_{m,\ell}} \right) \geq \E \left( \frac{\sum \sum w_{m ,\ell} y_{m,\ell}}{\sum \sum w_{m,\ell}} \   \right| \  A \right) \P(A)} \nonumber \\
&  \geq  &\E \left( \left. \frac{\sum \sum w_{m ,\ell} y_{m,\ell}}{\sum np_{n,\ell} +n^{.66} \delta_n } \  \right|  \  A \right) \P(A)  \geq \E \left( \frac{\sum \sum w_{m ,\ell} y_{m,\ell}}{\sum np_{n,\ell} +n^{.66} \delta_n } \right)- P(A^c), \nonumber
\end{eqnarray}
where for the last inequality we have used the fact that the risk of SNLM is bounded by 1. Since from the construction of SNLM in  \eqref{eq:$NL_1$}, $w_{m,\ell}$ is independent of $z_{m,\ell}$, we have
  \begin{eqnarray}  
\lefteqn{  \E \left( \frac{\sum \sum w_{m ,\ell} y_{m,\ell}}{\sum np_{n,\ell} +n^{0.66} \delta_n } \right)- P(A^c) = \E \left(\frac{\sum \sum w_{m ,\ell}x_{m,\ell}}{\sum np_{n,\ell} +n^{0.66} \delta_n }\right)- P(A^c)} \nonumber \\
&= &  \frac{\sum_{\ell< j^*} np_{n,\ell} }{\sum np_{n,\ell} +n^{0.66} \delta_n }- P(A^c)  \nonumber \geq   \frac{\sum_{\ell< j^*} np_{n,\ell} }{n+2 \sum_{\ell< j^*} np_{n,\ell} +n^{0.66} \delta_n }- P(A^c).
\end{eqnarray}
 Proposition \ref{thm:belowedge} proves that both the numerator $\sum_{\ell< j^*} np_{n,\ell}$ and the denomenator $\sum np_{n,\ell} +n^{0.66} \delta_n$ are $\Omega(n)$. 
Therefore, according to $A3$, we can ignore the summations $\sum_{\ell< j^*- \frac{\delta_n}{2}} np_{n,\ell} $ and $\sum_{\ell> j^*+ \frac{\delta_n}{2}} np_{n,\ell} $. By combining this fact with Lemma \ref{lem:colsymmetry}, we obtain
\begin{eqnarray*}
 \lefteqn{\frac{\sum_{\ell< j^*} np_{n,\ell} }{\sum np_{n,\ell} +n^{0.66} \delta_n }- P(A^c) \geq   \frac{\sum_{\ell< j^*} np_{n,\ell} }{np_{n,j^*}+2 \sum_{\ell< j^*} np_{n,\ell} +n^{0.66} \delta_n }- P(A^c)}\\ 
 &\geq&  \frac{\sum_{\ell< j^*} np_{n,\ell} }{n+2 \sum_{\ell< j^*} np_{n,\ell} +n^{0.66} \delta_n }- P(A^c). \hspace{5cm}
\end{eqnarray*}

In the last inequality we assumed that $p_{n,j^*} =1$. To find a lower bound for $\frac{\sum_{\ell< j^*} np_{n,\ell} }{n+2 \sum_{\ell< j^*} np_{n,\ell} +n^{0.66} \delta_n }$ it is enough to note that $ \frac{\sum_{\ell< j^*} np_{n,\ell} }{n+2 \sum_{\ell< j^*} np_{n,\ell} +n^{0.66} \delta_n }$ is an increasing function of
$\sum_{\ell< j^*} np_{n,\ell}$ and therefore is minimized if and only if $\sum_{\ell< j^*} np_{n,\ell}$ takes its minimum value. However, according to Proposition \ref{thm:belowedge} the minimum value of this term is $\Theta(n)$. Therefore, we have
\begin{eqnarray*}
\lefteqn {\frac{\sum_{\ell< j^*} np_{n,\ell} }{n+2 \sum_{\ell< j^*} np_{n,\ell} +n^{.66} \delta_n }- P(A^c) } \\
& \geq & \frac{np_0}{np_0+n +n^{.66} \delta_n } - p(A^c)  = \frac{p_0}{p_0+1}(1+o(1)). \hspace{1cm}
\end{eqnarray*}
This completes the proof.  
\end{proof}

\section{Tapered NLM Weights}\label{app:taperedweights}

In this section we show that the upper bound we provided for the NLM in Theorem \ref{thm:upperbound} holds in the more general setting of tapered weights. 
We now allow the weights to be a smooth function of the $\delta_n$-neighborhood. We assume that the weight assignment policy satisfies the following properties:

\begin{itemize}
\item[B1:] The neighborhood size $\delta_n= 2\log(n)$.
\item[B2:] The weights are non-negative and bounded, i.e., $ 0 \leq w_{m,\ell} \leq \alpha$.
\item[B3:] If $d^2 (x_{i,j},x_{m,\ell})=0$, then the assigned weight satisfies $\E(w_{i,j}(m,\ell)) > c$, for some constant $c$ independent of $n$. 
\item[B4:] If $d^2 (x_{i,j},x_{m,\ell})=1$, then $\E(w_{i,j}(m,\ell) ) = O\left(\frac{1}{\sqrt{n}}\right)$. It shall be emphasized that slower decay rate in this
expectation, results in slower decay in the rate of the NLM algorithm. 
\end{itemize}

\begin{thm}\label{thm:taperedweights}
If the weight assignment policy in NLM satisfies properties B1--B4, then
\[
\sup_{f \in H^{\alpha} (C)}R(f, \hat{f}^{N}) = O\left(\frac{\log n}{n}\right).
\]
\end{thm}

\begin{proof}
Consider the four partitions $S_1$--$S_4$ defined in the proof of Theorem \ref{thm:upperbound}. Our goal is to obtain an upper bound for the risk of the pixels in each region. The risk of the pixels in $S_2$ and $S_3$ will be bounded by the strategy we  employed in Theorem \ref{thm:upperbound}. Here, we just explain how we bound the risk of the pixels in $S_1$ and $S_4$. Since the proof for $S_4$ is the same as the proof for $S_1$, we consider just $S_1$. Let $(i,j) \in S_1$. Therefore $x_{i,j} = 0 $ and 
\begin{eqnarray}\label{eq:taper1}
\E(x_{i,j} -\hat{f}^{N}_{i,j})^2 &=& \E \left(\frac{\sum w_{m,\ell} y_{m,\ell}}{\sum w_{m,\ell}}\right)^2  \nonumber \\
&= &  \E \left(\frac{\sum w_{m,\ell} x_{m,\ell}}{\sum w_{m,\ell}}\right)^2 +  \E \left(\frac{\sum w_{m,\ell} z_{m,\ell}}{\sum w_{m,\ell}}\right)^2 \nonumber \\ 
&  & +~ \E \left(\left(\frac{\sum w_{m,\ell} x_{m,\ell}}{\sum w_{m,\ell}}\right)\left(\frac{\sum w_{m,\ell} z_{m,\ell}}{\sum w_{m,\ell}}\right)\right). 
\end{eqnarray}
To obtain an upper bound for the risk, we will find upper bounds for the last three terms in \eqref{eq:taper1}. Lemmas \ref{lem:biasbound_taper} and \ref{lem:varterm:taper} below summarize the upper bounds.

\begin{lem}\label{lem:biasbound_taper}
Let $w_{m,\ell}$ be the weights of NLM satisfying B1--B4. Then
\[
 \E \left(\frac{\sum w_{m,\ell} x_{m,\ell}}{\sum w_{m,\ell}}\right)^2 =O\left(\frac{1}{n} \right).
\]
\end{lem}
\begin{proof}
Define $S_f$ as the set of the indices of the pixels whose noise-free value is neither $0$ nor $1$,
and plug in the actual values of $x_{m,\ell}$ to obtain
\begin{eqnarray} \label{upper_bias_taper}
\lefteqn{\E \left(\frac{\sum_{(m,\ell) \in S_1 \cup S_4} w_{m,\ell} x_{m,\ell} + \sum_{(m,\ell) \in S_2 \cup S_3 \backslash S_f} w_{m,\ell} x_{m,\ell}+ \sum_{(m,\ell) \in S_f} w_{m,\ell} x_{m,\ell} }{\sum_{(m,\ell) \in S_1 \cup S_4}  w_{m,\ell}+ \sum_{(m,\ell) \in S_2 \cup S_3\backslash S_f}  w_{m,\ell}+ \sum_{(m,\ell) \in S_f}  w_{m,\ell}}\right)^2} \nonumber \\
  & \leq &  {\E \left(\frac{\sum_{(m,\ell) \in S_4} w_{m,\ell}  + \sum_{(m,\ell) \in S_3\backslash S_f} w_{m,\ell} +  \sum_{(m,\ell) \in S_f} w_{m,\ell} x_{m,\ell}}{\sum_{(m,\ell) \in S_1 \cup S_4}  w_{m,\ell}+ \sum_{(m,\ell) \in S_2 \cup S_3}  w_{m,\ell} +\sum_{(m,\ell) \in S_f} w_{m,\ell} }\right)^2} \nonumber \hspace{1.5cm}\\
  & \leq &  {\E \left(\frac{\sum_{(m,\ell) \in S_4} w_{m,\ell}  + \sum_{(m,\ell) \in S_3} \alpha +2n\alpha}{\sum_{(m,\ell) \in S_1 \cup S_4}  w_{m,\ell}+ \sum_{(m,\ell) \in S_3}  \alpha }\right)^2}.  
\end{eqnarray}
To derive the last inequality we use the following facts, which are easy to check:
\begin{enumerate}
\item $\left(\frac{\sum_{(m,\ell) \in S_4} w_{m,\ell}  + \sum_{(m,\ell) \in S_3\backslash S_f} w_{m,\ell} +  \sum_{(m,\ell) \in S_f} w_{m,\ell} x_{m,\ell}}{\sum_{(m,\ell) \in S_1 \cup S_4}  w_{m,\ell}+ \sum_{(m,\ell) \in S_2 \cup S_3}  w_{m,\ell} +\sum_{(m,\ell) \in S_f} w_{m,\ell} }\right)^2$ is an increasing function of $\sum_{(m,\ell) \in S_3 \backslash S_f} w_{m,\ell}$.
\item $\left(\frac{\sum_{(m,\ell) \in S_4} w_{m,\ell}  + \sum_{(m,\ell) \in S_3\backslash S_f} w_{m,\ell} +  \sum_{(m,\ell) \in S_f} w_{m,\ell} x_{m,\ell}}{\sum_{(m,\ell) \in S_1 \cup S_4}  w_{m,\ell}+ \sum_{(m,\ell) \in S_2 \cup S_3}  w_{m,\ell} +\sum_{(m,\ell) \in S_f} w_{m,\ell} }\right)^2$ is a decreasing function of $\sum_{(m,\ell) \in S_2 \backslash S_f} w_{m,\ell}$. 
\item $|S_f| \leq 2n$, i.e., $S_f$ contains at most $2n$ pixels.
\end{enumerate}

Our next claim is that $\sum_{(m,\ell) \in S_4} w_{m,\ell} $ and $\sum_{(m,\ell) \in S_1} w_{m,\ell}$ concentrate around their means.  We establish this in a manner very similar to the proof of Theorem \ref{thm:lowerbound}. We first break the $\sum_{(m,\ell) \in S_4} w_{m,\ell} $ into $(4\delta_n+2)^2$ subsequences such that each subsequece contains only independent random variables. In other words if $x_{m,\ell}$ is in one summation, then no other element of $\mathcal{C}^{4\delta_n+2}_{i,j}$ will be in the summation. Therefore, for each summation we can apply the Hoeffding inequality. Finally, we use the union bound as explained in the proof of Theorem \ref{thm:lowerbound} to show that 
\begin{align*}
&\P\left(\left|\sum_{(m,\ell) \in S_1} w_{m,\ell} - \sum_{(m,\ell) \in S_1} \E( w_{m,\ell}) \right| > t \right) \leq 2(4\delta_n+2)^2 {\rm e}^{\frac{- 2t^2}{(4\delta_n+2)^4(\sum_{(m,\ell) \in S_1} \alpha^2)}}, \\
&\P\left(\left|\sum_{(m,\ell) \in S_4} w_{m,\ell} - \sum_{(m,\ell) \in S_1} \E( w_{m,\ell}) \right| > t \right) \leq 2(4\delta_n+2)^2 {\rm e}^{\frac{- 2t^2}{(4\delta_n+2)^4(\sum_{(m,\ell) \in S_1} \alpha^2)}}.
\end{align*}
It is straightforward to prove that by setting $t$ to $32\alpha n \log^{2.5}(n)$, we have 
\begin{align}\label{conc_tapered}
&\P\left(\left|\sum_{(m,\ell) \in S_1} w_{m,\ell} - \sum_{(m,\ell) \in S_1} \E( w_{m,\ell}) \right| > 32 \alpha n \log^{2.5}(n) \right) \leq O\left( \frac{\delta_n^2}{n^{8}} \right) \nonumber \\
&\P\left(\left|\sum_{(m,\ell) \in S_4} w_{m,\ell} - \sum_{(m,\ell) \in S_1} \E( w_{m,\ell}) \right| > 32\alpha n \log^{2.5}(n) \right) \leq O\left( \frac{\delta_n^2}{n^{8}} \right).
\end{align}
Define the event $F$ as $\left\{\left|\sum_{(m,\ell) \in S_1} w_{m,\ell} - \sum_{(m,\ell) \in S_1} \E( w_{m,\ell}) \right| < 32\alpha n \log^{2.5}n \right\}$ $\cap$  $\left\{ |\sum_{(m,\ell) \in S_4} w_{m,\ell} - \sum_{(m,\ell) \in S_1} \E( w_{m,\ell}) | < 32\alpha n \log^{2.5}n \right\}$. It is clear from \eqref{conc_tapered} that 
\begin{equation} \label{complement_event_taper}
\P(\mathcal{F}^c) = O\left( \frac{\delta_n^2}{n^{8}} \right).
\end{equation}
Using \eqref{upper_bias_taper}, \eqref{conc_tapered}, and \eqref{complement_event_taper} we have
\begin{eqnarray*}
 \lefteqn{\E \left(\frac{\sum_{(m,\ell) \in S_4} w_{m,\ell}  + \sum_{(m,\ell) \in S_3 \backslash S_f} \alpha+ 2n\alpha}{\sum_{(m,\ell) \in S_1 \cup S_4}  w_{m,\ell}+ \sum_{(m,\ell) \in S_3\backslash S_f}  \alpha }\right)^2} \\
& \leq &
 \E \left( \left. \left(\frac{\sum_{(m,\ell) \in S_4} w_{m,\ell}  + \sum_{(m,\ell) \in S_3 \backslash S_f} \alpha + 2n\alpha}{\sum_{(m,\ell) \in S_1 \cup S_4}  w_{m,\ell}+ \sum_{(m,\ell) \in S_3 \backslash S_f}  \alpha}\right)^2 \ \right| F \right) \P(F) + \P(F^c) \\
 & \leq& O\left(\frac{1}{n}\right).
\end{eqnarray*}
The last inequality is due to Assumptions B3 and B4. This completes the proof of the lemma.
\end{proof}

\begin{lem} \label{lem:varterm:taper}
Let $w_{m,\ell}$ be the weights of NLM with $\delta_n = \log(n)$. Also assume that the weights are set according to $B1$-\hspace{.01cm}-$B4$. We then have
\[
 \E \left(\frac{\sum w_{m,\ell} z_{m,\ell}}{\sum w_{m,\ell}}\right)^2 = O\left(\frac{\log^2(n)}{n^2}\right).
\]
\end{lem}

\begin{proof}
We first condition on the event $F$ introduced in the proof of Lemma \ref{lem:biasbound_taper}. 
\begin{eqnarray*}
\lefteqn{ \E \left(  \left( \frac{\sum w_{m,\ell} z_{m,\ell} }{ \sum w_{m,\ell} } \right)^2 \right)}  \\
&\leq&  \E \left( \left.  \left( \frac{\sum w_{m,\ell} z_{m,\ell} }{ \sum w_{m,\ell} } \right)^2 \  \right|  \  F  \right) \P(F) + \P(F^c) \\
&\leq&  \E \left( \left.  \left( \frac{\sum w_{m,\ell} z_{m,\ell} }{ \sum \E (w_{m,\ell}) - 32\alpha n\log^{2.5}(n) } \right)^2 \  \right|  \  F \right) \P(F) + \P(F^c) \\
&\leq&  \E \left(  \left( \frac{\sum w_{m,\ell} z_{m,\ell} }{ \sum \E (w_{m,\ell}) - 32 \alpha n\log^{2.5}(n) } \right)^2 \right) + \P(F^c) \\
 &\leq& O\left( \frac{\log^2(n)}{n^2}\right).
\end{eqnarray*}
The last inequality is due to the fact that
\begin{eqnarray*}
 \lefteqn{\E\left( \left. \sum_{(m,\ell) \in S_{14}} w_{m, \ell} z_{m, \ell} \  \right| \  \mathcal{C}_{i,j}^{\delta_n} \right)^2} \\
 &=& \E\left( \left. \sum_{(m,\ell) \in S_{14}} \sum_{(m',\ell') \in S_{14}} w_{m ,\ell} z_{m,\ell} w_{m', \ell'} z_{m', \ell'} \  \right| \   \mathcal{C}_{i,j}^{\delta_n}  \right) \\
 &= & \sum_{(m,\ell) \in S_{14}} \sum_{(m',\ell') \in \mathcal{C}^{2\delta_n}_{m,\ell}} \E (w_{m ,\ell} z_{m,\ell} w_{m', \ell'} z_{m', \ell'} \  | \ \mathcal{C}_{i,j}^{\delta_n} ) =O(n^2 \delta_n^2). \hspace{1cm}
  \end{eqnarray*} 
Therefore, 
\begin{equation}
  \E \left(  \left( \frac{ \sum_{(m,\ell) \in S_{23}} w_{m,\ell}  z_{m,\ell}  }{ \sum_{(m,\ell) \in S_{14}} w_{m,\ell} }\right)^2\right) \leq O \left(\frac{\delta^2_n}{n^2} \right). 
\end{equation}
\end{proof}

Using the bounds derived in Lemmas \ref{lem:biasbound_taper}  and \ref{lem:varterm:taper}, we can complete the proof of the main theorem:
\begin{eqnarray*}
\lefteqn{\sup_{f \in H^{\alpha}(C)}  R(f, \hat{f}^{NL}) = \frac{1}{n^2} \sum_i \sum_j \E(x_{i,j} - \hat{f}^{N}_{i,j})^2} \\
 &\leq& \frac{\log^{2} (n) (|S_1| +|S_4|)}{n^4} + \frac{|S_2|+ |S_3|}{n^2} \leq  O\left(\frac{\log(n)}{n}\right).
\end{eqnarray*}
\end{proof}

\section{Discussion}

We have provided the first asymptotic result on the risk analysis of the nonlocal means (NLM) algorithm on smooth images with sharp edges.  In contrast to most other filtering approaches, NLM does not consider the spatial vicinity of the pixels as a feature for setting the weights. Instead, it exploits more global features, which leads to improved performance. 

In spite of this success, we have shown that the performance of NLM is within a logarithmic factor of the performance of the wavelet thresholding and still significantly below the optimal achievable rate. This is due to the fact that the isotropic nature of the NLM neighborhoods does not allow the
algorithm to discriminate the pixels that are close to but below the edge from the pixels that are close to but above the edge.  This leads to a blurring effect that results in high bias along the edge. Exploring the performance of NLM with anisotropic neighborhoods may address this issue and is left for future research.

\section{Acknowledgements}

This work was supported by the Grants NSF CCF-0431150, CCF-0728867, CCF-0926127, 
DARPA/ONR N66001-08-1-2065, N66001-11-1-4090, N66001-11-C-4092, 
ONR N00014-08-1-1112, N00014-10-1-0989, 
AFOSR FA9550-09-1-0432, 
ARO MURI W911NF-07-1-0185 and MURI W911NF-09-1-0383, 
and by the Texas Instruments Leadership University Program.

%\section{NOT SURE WEHRE TO PUT THESE YET}
%
%\subsection{Other related work}
%
%\textcolor{red}{1. It is still not complete. we need to search over the papers that have cited Tsybakov's book and see if there is any related work. \\}
%\textcolor{red}{2.This section shall be rewritten!} \\
%
%Despite the success of the NLM in applications, the theoretical analysis of its performance has been very limited. Aujol et al. \cite{Aujol:2009p3347} have analyzed exemplar based methods using geometric analysis on BV image models. Singer et al.  \cite{Singer:2009zr} analyze the bias of nonlocal filters on 1D signals from a theoretical physics perspective. \cite{Kervrann:2007p215,Raphan:2010qy} analyzeNLM using bayesian methods. However, these theoretical results do not provide any information on the performance of NLM on edges. \\

\bibliographystyle{model1a-num-names}
\bibliography{nonlocal,estimation,wedgelet}

\begin{thebibliography}{16}
\expandafter\ifx\csname natexlab\endcsname\relax\def\natexlab#1{#1}\fi
\providecommand{\bibinfo}[2]{#2}
\ifx\xfnm\relax \def\xfnm[#1]{\unskip,\space#1}\fi
%Type = Article
\bibitem[{Rudin et~al.(1992)Rudin, Osher, and Fatemi}]{RuOsFa92}
\bibinfo{author}{L.~I. Rudin}, \bibinfo{author}{S.~Osher},
  \bibinfo{author}{E.~Fatemi}, \bibinfo{journal}{Physica D: Nonlinear
  Phenomena} \bibinfo{volume}{60} (\bibinfo{year}{1992}) \bibinfo{pages}{259 --
  268}.
%Type = Article
\bibitem[{Donoho and Johnstone(1998)}]{Donoho:1998p1561}
\bibinfo{author}{D.~L. Donoho}, \bibinfo{author}{I.~M. Johnstone},
  \bibinfo{journal}{Annals of Statistics} \bibinfo{volume}{26}
  (\bibinfo{year}{1998}) \bibinfo{pages}{879--921}.
%Type = Book
\bibitem[{Korostelev and Tsybakov(1993)}]{Tsybakov:1993uq}
\bibinfo{author}{A.~Korostelev}, \bibinfo{author}{A.~Tsybakov},
  \bibinfo{title}{Minimax theory of Image Reconstruction}, Lecture Notes in
  Statistics, \bibinfo{publisher}{Springer-Verlag}, \bibinfo{year}{1993}.
%Type = Article
\bibitem[{Donoho(1999)}]{Donoho:1999p1950}
\bibinfo{author}{D.~L. Donoho}, \bibinfo{journal}{Annals of Statistics}
  \bibinfo{volume}{27} (\bibinfo{year}{1999}) \bibinfo{pages}{859 -- 897}.
%Type = Article
\bibitem[{Arias-Castro and Donoho(2002)}]{CaDo09}
\bibinfo{author}{E.~Arias-Castro}, \bibinfo{author}{D.~L. Donoho},
  \bibinfo{journal}{Annals of Statistics} \bibinfo{volume}{37}
  (\bibinfo{year}{2002}) \bibinfo{pages}{1172--1206}.
%Type = Book
\bibitem[{Yaroslavsky(1985)}]{Yaroslavsky85}
\bibinfo{author}{L.~Yaroslavsky}, \bibinfo{title}{Digital image processing-An
  introduction}, \bibinfo{publisher}{Springer Verlag}, \bibinfo{year}{1985}.
%Type = Article
\bibitem[{Smith and Brady(1997)}]{SmBr97}
\bibinfo{author}{S.~M. Smith}, \bibinfo{author}{J.~M. Brady},
  \bibinfo{journal}{International Journal of Computer Vision}
  \bibinfo{volume}{23} (\bibinfo{year}{1997}) \bibinfo{pages}{45--78}.
%Type = Inproceedings
\bibitem[{Tomasi and Manduchi(1998)}]{ToMa98}
\bibinfo{author}{C.~Tomasi}, \bibinfo{author}{R.~Manduchi}, in:
  \bibinfo{booktitle}{International Conference on Computer Vision}, pp.
  \bibinfo{pages}{839 --846}.
%Type = Article
\bibitem[{Lee(1983)}]{Lee83}
\bibinfo{author}{J.-S. Lee}, \bibinfo{journal}{Computer Vision, Graphics, and
  Image Processing} \bibinfo{volume}{24} (\bibinfo{year}{1983})
  \bibinfo{pages}{255 -- 269}.
%Type = Book
\bibitem[{Mallat(1997)}]{Mal97}
\bibinfo{author}{S.~Mallat}, \bibinfo{title}{A Wavelet Tour of Signal
  Processing}, \bibinfo{publisher}{Academic Press}, \bibinfo{address}{San
  Diego, CA}, \bibinfo{year}{1997}.
%Type = Techreport
\bibitem[{Candes and Donoho(1999)}]{Donoho:2000p1676}
\bibinfo{author}{E.~Candes}, \bibinfo{author}{D.~L. Donoho},
  \bibinfo{title}{{Curvelets: A Surprisingly Effective Nonadaptive
  Representation of Objects with Edges}}, \bibinfo{type}{Technical Report},
  \bibinfo{year}{1999}.
%Type = Article
\bibitem[{Kutyniok and Labatte(2007)}]{KuLa07}
\bibinfo{author}{G.~Kutyniok}, \bibinfo{author}{D.~Labatte},
  \bibinfo{journal}{Journal on Wavelet Theory and Applications}
  \bibinfo{volume}{1} (\bibinfo{year}{2007}) \bibinfo{pages}{1--10}.
%Type = Article
\bibitem[{Do and Vetterli(2005)}]{DoVe05}
\bibinfo{author}{M.~N. Do}, \bibinfo{author}{M.~Vetterli},
  \bibinfo{journal}{IEEE Transactions on Image Processing} \bibinfo{volume}{14}
  (\bibinfo{year}{2005}) \bibinfo{pages}{2091--2106}.
%Type = Article
\bibitem[{Buades et~al.(2005)Buades, Coll, and Morel}]{Buades:2005p27}
\bibinfo{author}{A.~Buades}, \bibinfo{author}{B.~Coll},
  \bibinfo{author}{J.~Morel}, \bibinfo{journal}{SIAM Journal on Multiscale
  Modelling and Simulation} \bibinfo{volume}{4} (\bibinfo{year}{2005})
  \bibinfo{pages}{490--530}.
%Type = Article
\bibitem[{Broder(2011)}]{Broder11}
\bibinfo{author}{J.~M. Broder}, \bibinfo{journal}{New York Times}
  (\bibinfo{year}{2011}).
%Type = Book
\bibitem[{Stein(1986)}]{Stein86}
\bibinfo{author}{C.~Stein}, \bibinfo{title}{Approximate computation of
  expectation}, \bibinfo{publisher}{Institute of Mathematical Statistics},
  \bibinfo{year}{1986}.

\end{thebibliography}


\begin{thebibliography}{10}
\providecommand{\url}[1]{\texttt{#1}}
\providecommand{\urlprefix}{URL }

\bibitem{CaDo09}
Arias-Castro, E., Donoho, D.L.: The curvelet transform for image denoising.
  Annals of Statistics  37(3),  1172--1206 (2002)

\bibitem{Aujol:2009p3347}
Aujol, J., Ladjal, S.: Exemplar-based inpainting from a variational point of
  view (Jan 2009), preprint

\bibitem{Buades:2005p27}
Buades, A., Coll, B., Morel, J.: A review of image denoising algorithms, with a
  new one. SIAM J. Multiscale Model. and Sim.  4(2),  490--530 (Jan 2005)

\bibitem{Donoho:2000p1676}
Candes, E., Donoho, D.L.: {Curvelets: A Surprisingly Effective Nonadaptive
  Representation of Objects with Edges}. Tech. rep. (1999)

\bibitem{DoVe05}
Do, M.N., Vetterli, M.: The contourlet transform: an efficient directional
  multiresolution image representation. IEEE Trans. on Image Proc.  14(12),
  2091--2106 (2005)

\bibitem{Donoho:1999p1950}
Donoho, D.L.: Wedgelets: Nearly minimax estimation of edges. Ann. Stat.  27(3),
   859 -- 897 (Jan 1999)

\bibitem{Donoho:1998p1561}
Donoho, D.L., Johnstone, I.M.: Minimax estimation via wavelet shrinkage. Ann.
  Stat.  26(3),  879--921 (Jan 1998)

\bibitem{Kervrann:2007p215}
Kervrann, C., Boulanger, J., Coupe, P.: Bayesian non-local means filter, image
  redundancy and adaptive dictionaries for noise removal. Lect. Notes Comp.
  Sci.  (Jan 2007)

\bibitem{Tsybakov:1993uq}
Korostelev, A., Tsybakov, A.: Minimax theory of Image Reconstruction. Lecture
  Notes in Statistics, Springer-Verlag (1993)

\bibitem{KuLa07}
Kutyniok, G., Labatte, D.: Construction of regular and irregular shearlet
  frames. J. Wavelet Theory and Appl.  1,  1--10 (2007)

\bibitem{Lee83}
Lee, J.S.: Digital image smoothing and the sigma filter. Computer Vision,
  Graphics, and Image Processing  24(2),  255 -- 269 (1983),
  \url{http://www.sciencedirect.com/science/article/B7GXG-4D8FS7F-2X/2/dca0500b89d839b8d67a0c34f6ed676c}

\bibitem{BaMaNa11}
Maleki, A., Narayan, M., Baraniuk, R.: Minimax analysis of the non-local means
  algorithm (2011), preprint

\bibitem{Mal97}
Mallat, S.: A Wavelet Tour of Signal Processing. Academic Press, San Diego, CA
  (1997)

\bibitem{Raphan:2010qy}
Raphan, M., Simoncelli, E.: An empirical bayesian interpretation and
  generalization of nonlocal means. Tech. rep., Courant Institute of
  Mathematical Sciences, New York University (October 2010)

\bibitem{RuOsFa92}
Rudin, L.I., Osher, S., Fatemi, E.: Nonlinear total variation based noise
  removal algorithms. Phys. D.  60(1-4),  259--268 (1992)

\bibitem{Singer:2009zr}
Singer, A., Schkolnisky, Y., Nadler, B.: Diffusion interpretation of nonlocal
  neighborhood filters for signal denoising. Siam J. Imag. Sci.  2(1),  118 --
  139 (2009)

\bibitem{SmBr97}
Smith, S.M., Brady, J.M.: SusanÑa new approach to low level image processing.
  International Journal of Computer Vision  23,  45--78 (1997),
  \url{http://dx.doi.org/10.1023/A:1007963824710}, 10.1023/A:1007963824710

\bibitem{ToMa98}
Tomasi, C., Manduchi, R.: Bilateral filtering for gray and color images. In:
  Computer Vision, International Conference on. pp. 839 --846 (Jan 1998)

\bibitem{Vershynin10}
Vershynin, R.: Non-asymptotic theory of random matrices:extreme singular
  values. Proc. of the intrnational congress of mathematicians  (2010)

\bibitem{Yaroslavsky85}
Yaroslavsky, L.: Digital image processing-An introduction. Springer Verlag
  (1985)

\end{thebibliography}






% \end{thebibliography}

\end{document}